\documentclass[11pt]{amsart}

\usepackage{a4wide}
\usepackage{verbatim}
\usepackage[mathscr]{eucal}
\usepackage[usenames]{color}
\usepackage[english]{babel}
\usepackage{graphicx}
\usepackage[hang]{subfigure}
\usepackage[utf8]{inputenc}

\theoremstyle{plain}
\begingroup
\newtheorem{theorem}{Theorem}[section]
\newtheorem{lemma}[theorem]{Lemma}
\newtheorem{proposition}[theorem]{Proposition}

\endgroup

\theoremstyle{definition}
\begingroup
\newtheorem{definition}[theorem]{Definition}
\newtheorem{remark}[theorem]{Remark}

\endgroup
 
\numberwithin{equation}{section}

\newcommand{\R}{{\mathbb R}}

\newcommand{\be}{\begin{equation}}
\newcommand{\ee}{\end{equation}}
\newcommand{\bes}{\begin{eqnarray}}
\newcommand{\ees}{\end{eqnarray}}

\newcommand{\setR}{\mathbb{R}}
\newcommand{\setN}{\mathbb{N}}
\newcommand{\prb}{\mathcal{P}}
\newcommand{\meas}{\mathfrak{M}_+}
\newcommand{\cadlag}{\mathrm{D}}

\newcommand{\dd}{\, \mathrm{d}} 
\newcommand{\diff}{\mathop{}\!\mathrm{d}}

 \newcommand{\sgn}{\operatorname{sgn}}
\newcommand{\supp}{\operatorname{supp}}
\newcommand{\e}{\mathrm{e}}

\usepackage{xspace}
\makeatletter
\newcommand\etc{etc\@ifnextchar.{}{.\@\xspace}}
\newcommand\ie{i.e.\@ifnextchar,{}{\@\xspace}}
\newcommand\eg{e.g.\@ifnextchar,{}{\@\xspace}}
\newcommand\wrt{w.r.t.\@ifnextchar,{}{\@\xspace}}
\makeatother

\begin{document}
\title[Power Repulsion and Attraction]{Asymptotic Behavior of Gradient Flows Driven by Nonlocal Power Repulsion and Attraction Potentials in One Dimension}
\author[M.~Di Francesco]{Marco Di Francesco}
\address{Marco Di Francesco \\ Mathematical Sciences \\ University of Bath \\ Claverton Down \\ Bath, BA2 7AY \\ United Kingdom}
\email{m.difrancesco@bath.ac.uk}
\author[M.~Fornasier]{Massimo Fornasier}
\address{Massimo Fornasier \\ Zentrum Mathematik \\ Technische Universit\"at M\"unchen \\ Boltzmannstra\ss e 3 \\ D-85747 Garching \\ Germany}
\email{massimo.fornasier@ma.tum.de}
\author[J.-C.~H\"utter]{Jan-Christian H\"utter}
\address{Jan-Christian H\"utter \\ Department of Mathematics \\ Massachusetts Institute of Technology \\ 77 Massachusetts Avenue \\ Cambridge, MA 02139-4307 \\ USA}
\email{huetter@mit.edu}
\author[D.~Matthes]{Daniel Matthes}
\address{Daniel Matthes \\ Zentrum Mathematik \\ Technische Universit\"at M\"unchen \\ Boltzmannstra\ss e 3 \\ D-85747 Garching \\ Germany}
\email{matthes@ma.tum.de}

\begin{abstract}
  We study the long time behavior of the Wasserstein gradient flow for an energy functional consisting of two components:
  particles are attracted to a fixed profile $\omega$ by means of an interaction kernel $\psi_a(z)=|z|^{q_a}$,
  and they repel each other by means of another kernel $\psi_r(z)=|z|^{q_r}$.
  We focus on the case of one space dimension and assume that $1\le q_r\le q_a\le 2$.
  Our main result is that the flow converges to an equilibrium if either $q_r<q_a$ or $1\le q_r=q_a\le4/3$,
  and if the solution has the same (conserved) mass as the reference state $\omega$.
  In the cases $q_r=1$ and $q_r=2$, we are able to discuss the behavior for different masses as well,
  and we explicitly identify the equilibrium state, which is independent of the initial condition.
  Our proofs heavily use the inverse distribution function of the solution.
\end{abstract}


\maketitle

\section{Introduction}
\label{cha:introduction}

\subsection{Setup and results}
We study existence, uniqueness and long-time behavior of non-negative weak solutions $\mu=\mu(t)$ of the PDE
\begin{equation}
  \label{eq:501}
  \partial_t\mu = \nabla \cdot \left[ \left(\nabla \psi_a \ast \omega - \nabla \psi_r \ast \mu \right)
    \mu\right], \quad \mu(0) = \mu_0,
\end{equation}
in one space dimension.
The functions $\psi_a,\psi_r : \setR \rightarrow \mathbb{R}$, which are given by
\begin{align}
  \label{eq:1}
  \psi_a(x) = |x|^{q_a}, \quad \psi_r(x)= |x|^{q_r},\quad \text{with parameters $q_a,\,q_r\in[1,2]$},
\end{align}
represent the \emph{attraction} and \emph{repulsion kernels}, respectively, 
and $\omega\in L^\infty(\setR)$ is a prescribed \emph{reference profile} of compact support.
Formally, equation \eqref{eq:501} is the gradient flow of the energy functional
\begin{equation}
  \label{contenergy}
  \mathcal{E}[\mu] = - \frac{1}{2} \int_{\mathbb R^d}\int_{\mathbb R^d} \psi_r(x-y)  \diff \mu(x) \diff \mu(y)  + \int_{\mathbb R^d} \int_{\mathbb R^d} \psi_a(x-y) \diff \omega(x) \diff \mu(y)
\end{equation}
with respect to the $L^2$-Wasserstein metric, for $d=1$.
A specific motivation to study $\mathcal E$ and its flow \eqref{eq:501} is given further below.
Although the flow could be considered in the general framework of measure solutions $\mu$,
we limit our analysis to absolutely continuous measures $\mu$ with bounded density of compact support.
Further, since \eqref{eq:501} is mass-preserving 
and invariant under simultaneous multiplication of $\mu$ and $\omega$ by a positive constant,
we may assume without loss of generality that $\mu(t)$ is a \emph{probability} measure.
By abuse of notation, we shall frequently identify $\mu$ with its density function.
\medskip

The heuristics behind the long-time behavior of $\mu(t)$  is the following.
In the \emph{attraction dominated case} $q_r<q_a$, the attracting force $\nabla\psi_a\ast\omega$ generated by the reference profile 
is stronger than the repulsion force $\nabla\psi_r\ast\mu$ between the particles on large space scales.
It is thus expected that particles remain at a bounded distance from $\omega$,
and that the system equilibrates at a localized stationary solution.
In the \emph{repulsion dominated case} $q_r>q_a$ the repulsive force dominates the attracting one at large distances 
and will cause initially ``sufficiently delocalized'' solutions to diverge as $t\to\infty$;
however, since attraction is stronger than repulsion at short distances,
there might be solutions that remain localized for all times.
In this paper, we are concerned with the \emph{attraction dominated} case $q_r<q_a$, 
and with the \emph{balanced} case $q_r=q_a$.

As a preliminary result, we prove well-posedness of the dynamics inside the set of probability measures with compactly supported and bounded density.
Since we are working in one space dimension, the evolution equation \eqref{eq:501} 
can be reformulated as an integro-differential equation for the inverse distribution function of $\mu$.
In that framework, we obtain existence and uniqueness using a fixed point argument, see Theorem \ref{thm:existence}.

The main part of the paper is then devoted to the long-time behavior of these solutions.
Depending on the value of the parameters $q_a$ and $q_r$, we obtain a more or less complete picture.
\begin{itemize}
\item For $q_r=q_a=2$, the behavior is threefold, see Theorem \ref{thm:q2}:
  if $\omega$ has mass larger than one, then $\mu(t)$ converges weakly to a Dirac measure concentrated at $\omega$'s center of mass;
  if $\omega $ is of mass smaller than one, then $\mu(t)$ converges vaguely to zero;
  finally, if $\omega$ is of unit mass, then each $\mu(t)$ is a translate of $\mu_0$,
  and $\mu(t)$'s center of mass converges to that of $\omega$ exponentially fast.
\item For $1=q_r<q_a\le2$, the solution $\mu(t)$ converges weakly to a steady state $\mu^\star$ 
  which is obtained by cutting off the positive function $\psi_a''\ast\omega$ in a symmetric way around its median to normalize its mass to one.
  See Theorem \ref{thm:q1} for details.
  In the more delicate case $q_a=1$, $\mu^\star$ is given by a symmetric cut-off of $\omega$ itself,
  provided that it has mass larger or equal to one;
  if $\omega$ is of smaller mass, then there exists no steady state (of unit mass), and $\mu(t)$ converges vaguely to $\omega$,
  losing excess mass towards infinity,
  see Theorem \ref{thm:q1bis}
\item If $1<q_r<q_a\le2$, or if $1<q_r=q_a<4/3$, and if $\omega$ has mass one, 
  then we prove weak convergence of $\mu(t)$ to some stationary solution $\mu^\star$, see Theorem \ref{thm:1q2}.
  Our proof is based on a compactness argument, and does not lead to an explicit characterization of the stationary state.
\end{itemize}
The detailed results in the special cases $q_r=1$ and $q_r=2$ are obtained by direct calculations with the inverse distribution function of $\mu(t)$.
The proof of equilibration for the more general situations with $1<q_r<q_a\le2$ or $1<q_r=q_a<4/3$ 
relies on subtle moment estimates for measures $\mu$ of finite energy.
These estimates have been recently derived by the second and third author,
and are published in a companion paper \cite{13-FornasierHuetter-VarProperties}.

\subsection{Motivation: image dithering}
Equation \eqref{eq:501} was introduced in \cite{FHS12} in the context of discrete variational methods for image dithering \cite{scgwbrwe11, testgwscwe11, grpost11}.
In particular, in \cite{testgwscwe11}, for $\psi_a(\cdot)=\psi_r(\cdot)=| \cdot|$ being the Euclidean norm, the authors considered the discrete energy functional
\begin{equation} \label{starting point}
  \mathcal E_N[p] := -\frac{1}{2 N^2} \sum_{i,j=1}^N  \psi_r(p_i - p_j) + \frac{1}{N}\sum_{i=1}^N \int_{\mathbb R^d} \omega(x) \psi_a(p_i- x) \diff x, 
\end{equation}
where the $N$ points $p_1,\ldots,p_N \in \mathbb R^d$ represent point masses that should approximate a given image $\omega$, 
which is assumed to be a compactly supported absolutely continuous probability measure.
The associated particle gradient flow $\partial_t p \in - \partial_p \mathcal E_N$ is of the form \eqref{eq:501}, where $\mu(t)$ is the sum of Dirac measures at the $p_j(t)$.
The minimizer of $\mathcal E_N$ places the $N$ points in an optimal way in the sense that 
points are concentrated on large values of the image $\omega$, but they are not too densely distributed. 
This behavior results from the balance between the second (attraction) term with the first (repulsion) term in the energy \eqref{starting point}.
It is shown that the energy functional is continuous and coercive, and, for $d=1$, it is possible to calculate its minimizers explicitly. 
In the two or higher dimensional settings, minimizers have been approximated numerically.
In \cite{grpost11} the data function $\omega$ was also considered 
on other sets than $\mathbb R^d$, such as $\mathbb T^d$ or $\mathbb S^2$, 
and kernels $\psi_a, \psi_r$ other than the Euclidean distance were used. 
When the number $N$ of particles is taken to $\infty$, one formally arrives at the continuous energy $\mathcal E$ given in \eqref{contenergy}.
The $\Gamma$-convergence of $( \mathcal E_N )_N$ to the lower-semicontinuous envelope of $\mathcal{E}$ with respect to the narrow topology 
has been shown in  \cite{13-FornasierHuetter-VarProperties}.

\subsection{Results from the literature}
\label{sec:well-posedness}
As indicated above, equation \eqref{eq:501} results as the mean-field limit of particle dynamics.
In \cite{FHS12} that limit process has been made rigorous under the well-known assumption $\psi_a,\psi_r\in C^{1,1}$.
Under this classical smoothness condition, the general gradient flow theory in probability spaces \cite{AGS08} is applicable to \eqref{eq:501},
yielding well-posedness of the initial value problem and contraction/expansion estimates on the flow.

If instead the repulsive kernel $\psi_r$ is of lower regularity (note that we only have $\psi_r\in C^{1,q_r-1}$), 
then well-posedness and particle approximation are difficult to analyze in general. 
Many results have been proven recently for \eqref{eq:501} and similar equations with non-smooth kernels.
The main focus has been on equations without external confinement ($\omega\equiv0$),
but with more general ``self-interaction'' kernels $\psi_r=-K$, i.e., 
the equation
\begin{align}
  \label{eq:K}
  \partial_t\mu = \nabla\cdot(\rho\,\nabla K\ast\rho)
\end{align}
is considered.
A typical choice for $K$ are combined repulsion-attraction kernels, 
that are repulsive at short distances but attractive at large distances.
We refer to \cite{13-Carrillo-Choi-Hauray-MFL} for a recent survey on the rapidly growing literature.
Below, we only mention very few selected contributions that are related to our own analysis.

In \cite{CDFLS}, a well-posedness theory has been developed mainly for attractive kernels $K$ 
that are symmetric and $C^1$-regular except possibly for the origin, and $\lambda$-convex
(see also \cite{CLM} for an extension to barely $\lambda$-convex potentials without any regularity assumption).
Existence and uniqueness are discussed for measure-valued solutions $\mu(t)$.
In fact, it is one of the key observations that despite the relatively high regularity imposed on $K$ by the $\lambda$-convexity assumption,
solutions $\mu$ to \eqref{eq:K} may generate point masses in finite time.
In \cite{belaro11}, convexity and smoothness hypotheses on $K$ have been relaxed,
and additional ``Osgood'' criteria have been formulated under which solutions remain absolutely continuous.
In this context, a solution theory in $L^p$ was developed.
This theory has been further extended in \cite{BCLR}, leading e.g.\ to existence and uniqueness of classical solutions
in the cases $K(x)\approx-|x|^\alpha$ with $\alpha>1$ near $x=0$ 
--- which corresponds to $\psi_r$ in \eqref{eq:1} with $1<q_r\le 2$ in our context.
The important special case $K(x)\approx-|x|$ has been analyzed in \cite{BS}, see also \cite{BCDP}.

In view of \cite{BCLR,BS}, our result on well-posedness of \eqref{eq:501} for the considered range of kernels (see Theorem \ref{thm:existence}) is not new.
Our motivation to include it here is that we also provide a short and quite elementary proof, based on the use of the inverse distribution function.
Our strategy of proof builds on that of \cite{BDi08}. 
Several non-trivial modifications are necessary to deal with the lack of smoothness of $\psi_r(x)$ at $x=0$.

Equilibration of solutions to equations of type \eqref{eq:501} in one spatial dimension has been investigated in a series of papers \cite{FR2,FR1,ra10};
an extremely rich theory in higher space dimensions is currently developing as well, see e.g.\ \cite{BCLR}.
Mainly, the shape and nonlinear stability of stationary states for \eqref{eq:K} have been analyzed 
for combined attractive-repulsive kernels $K$.
Under suitable hypotheses, there exist non-trivial steady states $\mu^\star$, even in the absence of external attraction forces.

The only result with direct relation to our own work is \cite[Theorem 1.2]{FR2}.
There, weak convergence of $\mu(t)$ towards a stationary state has been established (essentially) under the hypothesis 
that the kernel function $K:\setR\to\setR$ is convex on $\setR_+$ and can be written 
as the sum of a negative multiple of $|\cdot|$ and a $C^2$-smooth function.
This covers in particular the situation $\psi_r(z)=|z|$, i.e., $q_r=1$, considered here.
The essential improvement contained in our own result (see Theorem \ref{thm:q1}) is the explicit characterization of the steady state $\mu^\star$.

Long-time asymptotics for the general case of power-type kernels with exponents $1<q_r\le q_a\le 2$ (see Theorem \ref{thm:1q2}) 
have apparently not been addressed in the literature before.

\subsection{Structure of the paper}
\label{sec:structure-paper}

This paper is organized as follows. 
In Section \ref{sec:pseudo-inverse-techn} we recall some basic facts about the representation of positive finite measures by means of pseudo-inverse functions. 
This representation is our key technical tool, and it also motivates our notion of transient solutions for \eqref{eq:501}. 
Section \ref{sec:existence-theory-linfty} is dedicated to the result of well-posedness of the initial-value problem; its proof is given in Appendix \ref{sec:proof-theor-refthm:l}. 
The core of the paper is Section \ref{sec:asymptotics}, where we study the long time asymptotics of solutions.
The abstract convergence argument for \( q_r = q_a \in [1, 4/3) \) relies on subtle moment bounds obtained in \cite{13-FornasierHuetter-VarProperties}; 
for the sake of self-containedness, the derivation of these bounds is sketched in Appendix \ref{sec:moment-bound-subl}.

\subsection*{Notations}
We write $\meas(\setR)$ for the space of non-negative finite measures on the real line $\setR$,
and $\prb(\setR)$, $\prb_p(\setR)$ denote the subspaces of probability measures, and of probability measures with finite $p$th moment, respectively.
We shall consider $\prb_p(\setR)$ as a metric space, equipped with the $L^p$-Wasserstein distance $W_p$,
see Lemma \ref{lem:17} below for a definition.
By abuse of notation, we will frequently identify absolutely continuous measures with their density functions.
Two concepts of (weak) convergence on $\meas(\setR)$ will be used:
$\mu_n\to\mu$ \emph{narrowly} means that
\begin{align}
  \label{eq:narrow}
  \int_\setR\varphi(x)\dd\mu_n(x)\to\int_\setR\varphi(x)\dd\mu(x)  
\end{align}
holds for every bounded $\varphi\in C^0(\setR)$,
and $\mu_n\to\mu$ \emph{vaguely} means that \eqref{eq:narrow} is true for all $\varphi\in C^0(\setR)$ of compact support.
Note that $\mu_n(\setR)\to\mu(\setR)$ if $\mu_n\to\mu$ narrowly, but not necessarily if $\mu_n\to\mu$ only vaguely. 

Lebesgue spaces are denoted by $L^p(\setR)$, with $p\in[1,\infty]$, 
and $L^\infty_c(\setR)$ is the space of essentially bounded functions with compact support.
Finally, we denote by $\cadlag(I)$ the linear space of \emph{c\`{a}dl\`{a}g functions} over the interval $I\subseteq\setR$,
where either $I=\setR$ or $I=[a,b)$.
That is, $\cadlag(I)$ is the space of bounded functions $X:I\to\setR$, which are right continuous and have a left limit at every point. 
We endow $\cadlag(I)$ with the $\sup$-norm (\emph{not} the Skorokhod topology), which makes it a Banach space.

\section{The Pseudo-inverse}
\label{sec:pseudo-inverse-techn}

\subsection{Definition and elementary properties}
\label{sec:pseudo-inverse-elem}
In one spatial dimension, we can exploit a special transformation technique which makes equation \eqref{eq:501} much more amenable to estimates in the Wasserstein distance. More precisely, this distance can be explicitly computed in terms of pseudo-inverses.
\begin{definition}[CDF and Pseudo-Inverse]
  \label{def:cdf-pseudo-inverse}
  Given a non-negative finite measure $\mu\in\meas(\setR)$ on the real line, 
  we define its \emph{cumulative distribution function (CDF)} $F_\mu:\setR\to[0,\mu(\setR)]$ by
  \begin{equation*}
    F_\mu(x) := \mu((-\infty,\,x]), 
  \end{equation*}
  and its \emph{pseudo-inverse} $X_\mu:[0,\mu(\setR))\to\setR$ by
  \begin{equation*}
    X_\mu(z) := \inf \left\{x \in \mathbb{R} : F_\mu(x) > z\right\}.
  \end{equation*}
\end{definition}
By definition, $F_\mu$ is a c\`{a}dl\`{a}g function, and so is $X_\mu$, i.e., $X_\mu\in\cadlag([0,\mu(\setR)))$.
Note that in some cases, $X_\mu$ indeed is an inverse of $F_\mu$. 
Namely, if $F_\mu$ is \emph{strictly} monotonically increasing, corresponding to \(\mu\) having its support on the whole of $\mathbb{R}$, then \(X_\mu \circ F_\mu = \mathrm{id}\). 
If \(F\) is \emph{continuous}, which means that $\mu$ does not give mass to points, then \(F_\mu\circ X_\mu = \mathrm{id}\). 
However, in general we only have
\begin{equation}
  \label{eq:219}
  (X_\mu\circ F_\mu) (x) \geq x, \ x \in \mathbb{R}, \quad (F_\mu\circ X_\mu)(z) \geq z, \ z\in[0,\mu(\setR)).
\end{equation}
\begin{lemma}[Substitution formula]
  \label{lem:16} Given $\nu\in\meas(\setR)$, then, for all $f \in L^1(\setR;\nu)$,
  \begin{equation*}
    \int_\mathbb{R} f(x) \, \mathrm{d} \nu(x) = \int_0^{\nu(\setR)} f(X(z)) \, \mathrm{d} z.
  \end{equation*}  
\end{lemma}
A direct consequence of the substitution formula is the following convenient representation for the convolution of functions with measures.
\begin{lemma}[Representation of the convolution]
  If $\psi:\setR\to\setR$ is a continuous function with at most quadratic growth of $|\psi(x)|$ for $x\to\pm\infty$,
  and if $\nu\in\prb_2(\setR)$ has pseudo-inverse $X_\nu$,
  then 
  \begin{align}
    \label{eq:2}
    \big(\psi\ast\nu\big)(x) := \int_\setR \psi(x-y)\dd\nu(y) = \int_0^{\nu(\setR)} \psi(x-X_\nu(\zeta))\dd\zeta \quad \text{for every $x\in\setR$}.
  \end{align}
\end{lemma}
\begin{lemma}[Formula for the Wasserstein-distance]
  \label{lem:17} \cite[Section 2.2]{07_Carillo_Toscani_prob-metrics} 
  Let $\mu,\nu\in\meas(\setR)$ with $m:=\mu(\setR)=\nu(\setR)$.
  Then, for \(p \in [1,\infty]\), the \emph{$p$th Wasserstein distance} between $\mu$ and $\nu$ equals
  \begin{equation*}
    W_p(\mu,\nu) = \left\| X_\mu-X_\nu\right\|_{L^p} =
    \begin{cases}
      \left( \int_0^m \lvert X_\mu(z) - X_\nu(z)\rvert^p \, \mathrm{d} z \right)^{1/p} & 1 \leq p < \infty,\\
      \sup_{0<z<m} \left| X_\mu(z) - X_\nu(z) \right| & p = \infty.
    \end{cases}
  \end{equation*}
\end{lemma}

We refer to \cite{piro13} for a recent notion of generalized Wasserstein-distance between measures of {\it different} mass. Some of the results we obtain below
can be restated in terms of these generalized distances.

\subsection{The transformed equation}
\label{sec:pseudo-inverse-trans}

In order to transform equation \eqref{eq:501} in terms of the pseudo-inverse, we introduce further notations.
Denote by $\mu:[0,\infty)\to\prb_2(\setR)$ one of its solutions and by $\omega$ the given datum, 
as well as by $F(t,\cdot)=F_{\mu(t)}$ and $G=F_\omega$ the respective CDFs and by $X(t,\cdot)=X_{\mu(t)}$ and $Y=X_\omega$ their pseudo-inverses.
Recall that we required $\mu(t)$ to be a probability measure at every $t\ge0$, but we allow $m:=\omega(\setR)>0$ to be an arbitrary positive real, at least for a part of our results,
so $X(t;\cdot)$ and $Y$ are defined on different domains in general.

Let us further assume for now that equality holds in the inequalities \eqref{eq:219}. 
Then we can, at least formally, compute the derivatives of these identities.
From $F(t,X(t,z)) = z$, we get by differentiating with respect to time and space, respectively:
\begin{gather}
  \partial_t F(t, X(t,z)) + \partial_x F (t,X(t,z)) \cdot \partial_t X(t,z) = 0, \label{eq:222}\\
  \partial_x F(t, X(t,z)) \cdot \partial_z X(t,z) = 1. \label{eq:223}
\end{gather}
From \eqref{eq:222}, we obtain 
\begin{equation*}
  \partial_t X = \left( -(\partial_x F)^{-1} \cdot \partial_t F\right) \circ X.
\end{equation*}
Now we can integrate \eqref{eq:501} w.r.t.\ $x\in\setR$ to derive an equation for $\partial_t F$, 
namely
\begin{equation*}
  \partial_tF = (\psi_a' \ast \omega - \psi_r' \ast \mu)\mu,
\end{equation*}
where at the moment we interpret $\mu$ as a density. Using $\partial_x F = \mu$
and combining \eqref{eq:222} and \eqref{eq:223}, we see that
\begin{equation*}
  \partial_t X = - (\psi_a' \ast \omega - \psi_r' \ast \mu) \circ X .
\end{equation*}
In the case where $1<q_a\le2$ and $1<q_r\le2$, 
  the functions $\psi_a'$ and $\psi_b'$ are continuous with sublinear growth, so we can use the representation \eqref{eq:2}. 
This yields the formulation which we want to work with:
\begin{equation}\label{maineqM}
  \partial_t X(t,z) = - \int_0^m \psi_a'(X(t,z) - Y(\zeta)) \, \mathrm{d} \zeta + \int_0^1 \psi_r'(X(t,z) - X(t,\zeta)) \, \mathrm{d} \zeta.
\end{equation}
If instead $q_r=1$, then \(\psi_r' = \operatorname{sgn}\) is not continuous and hence \eqref{eq:2} is not directly applicable.
In that situation, we assume in addition that \(\mu(t)\) is absolutely continuous.
Since for absolutely continuous $\nu\in\meas(\setR)$, one has
\begin{align}
  \int_{\mathbb{R}} \psi'(x - y) \, \mathrm{d} \nu(y) = \int_\mathbb{R} \operatorname{sgn}(x - y) \, \mathrm{d} \nu (y) 
  = \nu((-\infty,x]) - \nu((x,\infty))
  = 2F_\nu(x) - \nu(\mathbb{R})\label{eq:328}
\end{align}
at every $x\in\setR$, it follows that the evolution equation \eqref{maineqM} attains the special form
\begin{align}
  \label{eq:evolq1}
  \partial_t X(t,z) = - \int_0^m \psi_a'(X(t,z) - Y(\zeta)) \, \mathrm{d} \zeta + 2z - 1.
\end{align}
And in the particular case $q_a=q_r=1$, the evolution equation \eqref{eq:501} simplifies to
\begin{equation}
  \label{eq:229}
  \partial_{t}X(t,z) = 2 \left[z - G(X(t,z))\right]+m-1.
\end{equation}

These formal calculations motivate our definition of solutions for \eqref{eq:501} in the next section.

\section{Existence of solutions}
\label{sec:existence-theory-linfty}
Recall our choice \eqref{eq:1} for $\psi_a$, $\psi_r$, and let $\omega\in\meas(\setR)$ be a given profile with total mass $m:=\omega(\setR)>0$.
We further assume that $\omega$ is absolutely continuous with a density function in $L^\infty_c$, so that $Y:=X_\omega\in\cadlag([0,m))$.
\begin{definition}
  \label{dfn:transient}
  A map $\mu:[0,\infty)\to\prb_2(\setR)$ is called a \emph{transient solution} to \eqref{eq:501},
  if it satisfies the following:
  \begin{itemize}
  \item $\mu(t)$ has a density in $L^\infty_c$ at every time $t\ge0$,
  \item the map $t\mapsto X_{\mu(t)}$ is continuously differentiable from $[0,\infty)$ into $\cadlag([0,1))$,
  \item the evolution equation
    \begin{align}
      \label{eq:evol}
      \begin{split}
        \partial_t X_{\mu(t)}(z) = 
        & - 
        \begin{cases}
          \int_0^m\psi_a'\big(X_{\mu(t)}(z)-Y(\zeta)\big)\dd\zeta & \text{if $q_a>1$}, \\
          2G\big(X_{\mu(t)}(z)\big)-m & \text{if $q_a=1$},
        \end{cases}
        \\
        &\quad +
        \begin{cases}
          \int_0^1\psi_r'\big(X_{\mu(t)}(z)-X_{\mu(t)}(\zeta)\big)\dd\zeta & \text{if $q_r>1$}, \\
          2z-1 & \text{if $q_r=1$}
        \end{cases}
      \end{split} 
    \end{align}
    holds at every $t\ge0$ and every $z\in[0,1)$.
  \end{itemize}
\end{definition}
Note that continuous differentiability as a map into $\cadlag([0,1))$ implies that $t\mapsto X_{\mu(t)}(z)$ is a differentiable real function for every $z\in[0,1)$.
\begin{remark}
  In the cases $q_r=1$ or $q_r=2$, the energy functional $\mathcal E$ from \eqref{contenergy} is $\lambda$-convex for some $\lambda\in\setR$, 
  and therefore it possesses a unique $\lambda$-contractive gradient flow in the Wasserstein $W_2$ metric ---
  we refer to \cite{AGS08} for the case $q_r=2$, and to \cite{BS,BCDP} for the case $q_r=1$.
  It is easily seen that these gradient flow solutions are transient solutions in the sense of Definition \ref{dfn:transient} above, and vice versa.
\end{remark}
\begin{theorem}[Existence of solutions]
  \label{thm:existence}
  Let an initial condition $\mu_0\in\prb(\setR)$ with density in $L^\infty_c(\mathbb{R})$ be given.
  Then, there is a unique transient solution $\mu:[0,\infty)\to\prb_2(\setR)$ in the sense of Definition \ref{dfn:transient} with $\mu(0)=\mu_0$.
  In particular, $\mu$ is a distributional solution of the original equation \eqref{eq:501}, i.e., 
  for all $\varphi \in C^{\infty}_{c}([0,\infty) \times \mathbb{R})$, it satisfies
  \begin{align}
    \label{eq:weak} 
    \begin{split}
      &-\int_{0}^{\infty} \int_{\mathbb{R}} \partial_{t}\varphi(t,x) \, \mathrm{d} \mu(t,x) \, \mathrm{d} t - \int_{\mathbb{R}} \varphi(0,x) \, \mathrm{d} \mu_0(x)\\
      ={}&\int_{0}^{\infty} \int_{\mathbb{R}} \partial_{x}\varphi(t,x) \cdot \left(\psi_r' \ast \mu(t,.)\right)(x) \, \mathrm{d} \mu(t,x) \, \mathrm{d} t \\
      &- \int_{0}^{\infty} \int_{\mathbb{R}} \partial_{x} \varphi(t,x) \cdot \left(\psi_a' \ast \omega\right)(x)\, \mathrm{d} \mu(t,x) \, \mathrm{d} t. 
    \end{split}
  \end{align}
\end{theorem}
We stress that Theorem \ref{thm:existence} is not the main result of the paper.
In fact, it can be obtained by combining results from the recent literature:
the case $1<q_r\le2$ is covered by \cite[Theorem 7]{BCLR}, and $q_r=1$ is contained in \cite[Theorem 4.3.1]{BS}.
However, we provide a simple and short unifying proof of this result (see Appendix \ref{sec:proof-theor-refthm:l})
which is based on the use of the inverse distribution function.

\section{Asymptotic behavior}
\label{sec:asymptotics}
This section contains our main contribution: an analysis of the long time asymptotics in various regimes of $q_a$ and $q_r$.
First, we give a definition of \emph{steady states} in the spirit of Section \ref{sec:existence-theory-linfty}.
\begin{definition}
  \label{dfn:steady}
  A probability measure $\mu^\star \in\prb(\setR)$ with density in $L^\infty_c(\setR)$ is called a \emph{steady state} for the equation \eqref{eq:evol} and exponents $1\leq q_r, q_a \leq 2$,
  if its inverse distribution function $X^\star=X_{\mu^\star}$ satisfies 
 \begin{equation}
      \label{eq:evoloo}
        \begin{cases}
          \int_0^m\psi_a'\big(X^\star(z)-Y(\zeta)\big)\dd\zeta & \text{if $q_a>1$}, \\
          2G\big(X^\star(z)\big)-m & \text{if $q_a=1$},
        \end{cases}
        =
        \begin{cases}
          \int_0^1\psi_r'\big(X^\star(z)-X^\star(\zeta)\big)\dd\zeta & \text{if $q_r>1$}, \\
          2z-1 & \text{if $q_r=1$}
        \end{cases}
    \end{equation}
    at every  $z\in[0,1)$.
\end{definition}
We shall now study various regimes of the exponents $1\leq q_r\le q_a \leq 2$ in detail.

\subsection{The case $q_r = q_a = 2$}
\label{sec:case-q-2}
The most special case that we consider is $q_a = q_r = 2$. 
%
As $\psi_r'(x) = 2x$, equation \eqref{eq:evol} simplifies significantly:
\begin{align}
  \label{eq:233}
  \begin{split}
    \partial_t X(t,z) &= - \int_0^m 2(X(t,z) - Y(\zeta)) \, \mathrm{d} \zeta + \int_0^1 2(X(t,z) - X(t,\zeta)) \, \mathrm{d} \zeta \\
    &= -2(m-1)X(t,z) + 2\int_0^m Y(\zeta) \, \mathrm{d} \zeta - 2\int_0^1 X(t,\zeta)\, \mathrm{d} \zeta .
  \end{split}
\end{align}
This integro-differential equation can be easily solved as follows.
First, integrate \eqref{eq:233} \wrt $z\in(0,1)$ on both sides to obtain
\begin{align*}
  \frac{\mathrm{d}}{\mathrm{d} t} \bigg(\int_0^1 X(t,z)\dd z - \frac1m \int_0^m Y(z)\dd z\bigg)
  = - 2m \bigg(\int_0^1 X(t,\zeta)\dd\zeta - \frac1m \int_0^m Y(\zeta)\dd\zeta\bigg).
\end{align*}
Therefore,
\begin{align}
  \label{eq:666}
  \int_0^1 X(t,\zeta)\dd\zeta - \frac1m \int_0^m Y(\zeta)\dd\zeta = e^{-2mt}\bigg(\int_0^1 X(0,\zeta)\dd\zeta - \frac1m \int_0^m Y(\zeta)\dd\zeta \bigg).
\end{align}
Since $\int_0^1X(t,z)\dd z$ and $\frac1m\int_0^mY(z)\dd z$ equal to the center of mass of $\mu(t)$ and of $\omega$, respectively,
it follows that the center of mass of $\mu(t)$ converges to that of $\omega$ with the exponential rate $\exp(-2mt)$.
Now insert \eqref{eq:666} into \eqref{eq:233} and solve for $X(t,z)$:
\begin{align*}
  X(t,z) = e^{-2(m-1)t}\bigg(X(0,z) - \big(1-e^{-2t}\big) \int_0^1X(0,\zeta) \dd\zeta \bigg)  + \frac{1-e^{-2mt}}{m}\int_0^mY(\zeta)\dd\zeta .
\end{align*}
It is easily seen that $X(t,z)$ converges to the center of mass of $\omega$ at exponential rate $\exp(-2(m-1)t)$ for $m>1$, 
while for $0<m<1$, the value of $X(t,z)$ diverges to $+\infty$ or to $-\infty$ if $X(0,z)$ lies to the right or to the left of $\mu(0)$'s center of mass, respectively.
In the borderline case $m=1$, the functions $X(t,\cdot)$ at different times $t\ge0$ only differ by a constant.
In conclusion, we obtain the following.
\begin{theorem}
  \label{thm:q2}
  Assume $q_a=q_r=2$, and let $\mu:[0,\infty)\to\prb(\setR)$ be a transient solution.
  \begin{itemize}
  \item If $m>1$, then $\mu(t)$ converges weakly to a Dirac measure concentrated at $\omega$\!'s center of mass.
  \item If $0<m<1$, then $\mu(t)$ converges vaguely to zero.
  \item If $m=1$, then every $\mu(t)$ is a translate of $\mu_0$, and the center of mass of $\mu(t)$ converges to that of $\omega$ with exponential rate $\exp(-2mt)$.
  \end{itemize}
\end{theorem}

\subsection{The case $q_r = 1$}
\label{sec:case-q-1}
Next, we consider the case \( q_r = 1 \).
Then \( \psi_r \) is a multiple of the Newtonian potential in 1D, \ie $\psi_r''=2\delta_0$.
Definition \ref{dfn:steady} of stationary solution $\mu^\star$ specializes to the following equation for
its inverse distribution function $X^\star$:
\begin{equation}
  \label{eq:360}
  2z - 1 =
  \begin{cases}
    \int_0^m \psi_a'(X^\star(z) - Y(\zeta)) \, \mathrm{d} \zeta & \text{if $q_a>1$}, \\
    2G(X^\star(z))-m & \text{if $q_a=1$}.
  \end{cases}
\end{equation}
In Proposition \ref{prp:steady-stat-equat} below, 
we identify the unique steady state $\mu^\star$ of \eqref{eq:501} as 
a suitable cut-off of \( \frac{1}{2} \psi_a'' \ast \omega \) --- note that $ \frac{1}{2} \psi_a''\ast\omega=\omega$ if $q_a=1$. 
Theorem \ref{thm:q1} provides convergence of transient solutions $\mu(t)$ towards $\mu^\star$ as \( t\rightarrow\infty \).
\begin{proposition}[Existence and uniqueness of steady states]
  \label{prp:steady-stat-equat}
  Assume that either $q_a>1$, or that $q_a=1$ and $m\ge1$.
  Then, there are (unique, if $q_a>1$) real numbers $\underline x<x_0<\overline x$ such that
  \begin{align*}
    \psi'_a\ast\omega(x_0)=0, \quad \psi'_a\ast\omega(\underline x)=-1,\quad \psi'_a\ast\omega(\overline x)=1.
  \end{align*}
  And there is a unique steady state in the sense of Definition \ref{dfn:steady},
  which is given by
  \begin{align*}
    \mu^\star = \frac12\big(\psi_a''\ast\omega\big)\, \mathbf{1}_{[\underline x,\overline x]}.
  \end{align*}
  If instead $q_a=1$ and $m<1$, then there is no steady state --- but see Theorem \ref{thm:q1bis}.
\end{proposition}
\begin{figure}[t]
  \centering
  \subfigure[\( \psi_a' \ast \omega \) determines \( x_0,a=\underline x,b=\overline x \)]{\includegraphics[]{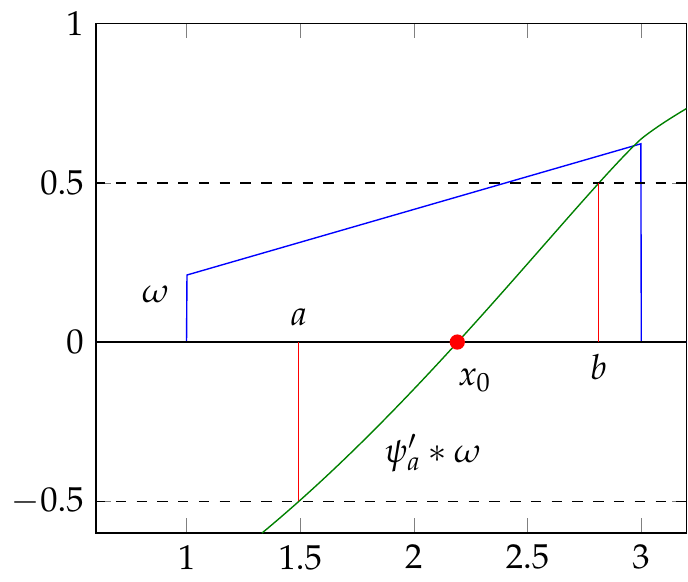}}
  \hspace{.5cm}
  \subfigure[\( \psi_a'' \ast \omega \) determines the steady state profile]{\includegraphics[]{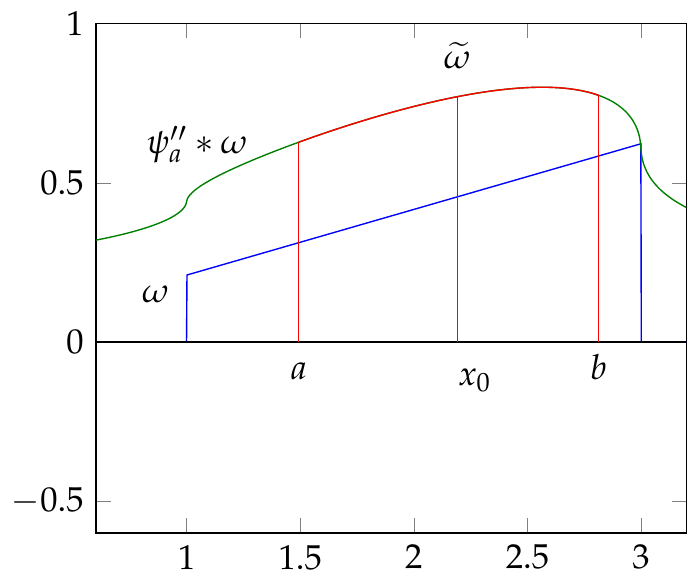}}
  \caption{Example of the construction of the steady state \( \widetilde{\omega} = \mu^\star\) for \( q_a = 1.5 \)}
  \label{fig:q1plot}
\end{figure}
See Figure \ref{fig:q1plot} for an example of the resulting \( \mu^\star \), using the notation of Proposition \ref{prp:steady-stat-equat}.
\begin{proof}
  We subdivide the proof in two cases depending on $q_a$.
  
  \emph{First, assume $q_a>1$.}
  Let $U:=\psi_a'\ast\omega$.
  Since $\psi_a'$ is strictly monotonically increasing, continuous, and tends to \( \pm \infty \) as \( x \rightarrow \pm \infty \), the same is true for $U$. 
  Hence $x_0,\underline x,\overline x$ are well-defined.
  Formula \eqref{eq:2} is applicable, and the steady state equation \eqref{eq:360} is thus equivalent to
  \begin{equation*}
    2z - 1 = U(X^\star(z)), \quad 0\le z<1.
  \end{equation*}
  The function $U$ is invertible, and so the unique solution $X^\star:[0,1]\to\setR$ is given by
  \begin{equation*}
    X^\star(z) = U^{-1}(2z - 1).
  \end{equation*}
  By the properties of $U$, the function $X^\star$ so defined is strictly increasing and bounded, 
  hence it is the inverse distribution function of some absolutely continuous measure $\mu^\star\in\prb(\setR)$ with density in $L^\infty_c$.
  The associated CDF \( F^\star=F_{\mu^\star} \) is given by
  \begin{equation*}
    F^\star(x) = 
    \begin{cases}
      0 & \text{if $U(x) < -1$},\\
      \frac{1}{2} \left( U(x) + 1 \right) & \text{if $U(x) \in [-1,1]$},\\
      1 & \text{if $U(x) > 1$}.
    \end{cases}
  \end{equation*}
  In conlusion, the steady state $\mu^\star$ has its median at the unique zero of $U$, which is $x_0$. 
  And its density coincides with \( \frac12U' = \frac12 \psi''_a\ast\omega \), extending from $x_0$ in both directions until mass $1/2$ is reached on each side.
  This proves the claim for $q_a>1$.

  \emph{Now assume $q_a=1$.} 
  The right-hand side of \eqref{eq:360} varies in $[-m,m]$.
  If $m<1$, then clearly there are values $z\in[0,1)$ for which \eqref{eq:360} cannot be true.
  Hence, in that case there is no steady state in the sense of Definition \ref{dfn:steady}.
  
  Let $m\ge1$ from now on.
  Since $\omega\in L^\infty_c$, its CDF $G$ is continuous and monotonically increasing, and $G(x)=m$, $G(-x)=0$ for all sufficiently large $x>0$.
  Hence for every $z\in[0,1)$, there is at least one $X^\star(z)\in\setR$ such that \eqref{eq:360} holds.
  Note that $G$ is constant on intervals where $\omega$ vanishes, so $X^\star(z)$ is not uniquely determined in general.
  However, it is easily seen that among all the possible solutions $X^\star:[0,1)\to\setR$, there is precisely one which is monotonically increasing and right continuous,
  namely $z\mapsto Y(z+(m-1)/2)$.
  There is an associated $\mu^\star\in\prb(\setR)$ with density in $L^\infty_c$ and corresponding CDF $F^\star:\setR\to[0,1]$ given by
  \begin{equation*}
    F^\star(x) = 
    \begin{cases}
      0 & \text{if $G(x) < (m - 1)/2$},\\
      G(x) - \frac{m - 1}{2} & \text{if $G(x) \in [(m-1)/2, \, 1 + (m-1)/2]$},\\
      1 & \text{if $G(x) > 1 + (m-1)/2$}.
    \end{cases}
  \end{equation*}
  Similarly as before, we conclude that $\mu^\star$ has the same median as $\omega$,
  and that the density of $\mu^\star$ coincides with $G'=\omega$ on an interval $[\underline x,\overline x]$ chosen such that $\mu^\star$ has unit mass.
\end{proof}
\begin{theorem}[Asymptotic stability]
  \label{thm:q1}
  Assume that either $q_a>q_r=1$, or that $q_a=q_r=1$ and $m\ge1$. 
  Let $\mu^\star$ be the unique stationary state determined in Proposition \ref{prp:steady-stat-equat} above,
  and let $\mu:[0,\infty)\to\prb(\setR)$ be a transient solution in the sense of Definition \ref{dfn:transient}.
  Then $W_2(\mu(t), \mu^\star)$ converges to zero monotonically as $t\to\infty$.
\end{theorem}
\begin{proof}
  Recall that $X(t,\cdot)=X_{\mu(t)}$ denotes the (time-dependent) inverse distribution function of a transient solution.
  Define $U=\psi_a'\ast\omega$ as in the proof of Proposition \ref{prp:steady-stat-equat},
  and recall further that the inverse distribution function $X^\star$ of $\mu^\star$ satisfies $U(X^\star(z))=2z-1$ for all $z\in[0,1)$.

  Fix some $z\in[0,1)$.
  By Definition \ref{dfn:transient} of transient solutions, $X$ is a continuously differentiable curve in $\cadlag([0,1))$.
  Hence $y:[0,\infty)\to\setR$ with $y(t)=X(t,z)$ is continuously differentiable,
  and satisfies --- in view of \eqref{eq:evol} --- the ordinary differential equation
  \begin{equation*}
    \dot y(t) = U(y^\star) - U(y(t)),
  \end{equation*}
  where $y^\star=X^\star(z)$.
  Since $U:\setR\to\setR$ is a Lipschitz continuous function (see Lemma \ref{lem:18}),
  different solution curves to this differential equations cannot cross, 
  so either $y(t)\le y^\star$ or $y(t)\ge y^\star$ consistently for all $t\ge0$.
  Since further $U$ is monotonically increasing, it follows that $y(t)$ is bounded and monotone in time ---
  increasing if $y(0)\le y^\star$ and decreasing if $y(0)\ge y^\star$.
  Thus $y(t)$ converges (monotonically) to some limit $y_\infty$ as $t\to\infty$.
  By continuity of $U$, it follows that $0=U(y_\infty)-U(y^\star)$.
  %
  %
  If $q_a>1$, then \emph{strict} monotonicity of $U$ directly implies $y_\infty=y^\star$.
  If $q_a=1$,  then
  \begin{align*}
    U(x) = - \int_{-\infty}^0 \omega(x-y)dy +\int_{0}^{+\infty} \omega(x-y)dy =1 -2 G(x) 
  \end{align*}
  is strictly increasing \emph{except} possibly for an at most countable number of intervals on which it is constant.
  By definition of the pseudo-inverse, $X^\star$ never takes values \emph{inside} these intervals.
  Instead, there is an at  most countable number of values $\tilde z\in(0,1)$ for which $X^\star(\tilde z)$ is the end point of one such interval.
  For all other (i.e., almost all) values $z\in(0,1)$, $U$ is strictly monotone at $X^\star(z)$.
  For each of those $z$, we conclude $y_\infty=y^\star$ as well.
  %

  In summary, for (almost, if $q_a=1$) all $z\in(0,1)$, the values $X(t,z)$ converge to their respective limit $X^\star(z)$ monotonically as $t\to\infty$.
  The monotone convergence theorem now implies that
  \begin{align*}
    W_2(\mu(t),\mu^\star)^2 = \int_0^1 |X(t,\zeta)-X^\star(\zeta)|^2\dd\zeta 
  \end{align*}
  tends to zero monotonically.
\end{proof}
\begin{remark}[No \( W_\infty \)-convergence]
  If \( q_r = q_a = m = 1 \), then \( X^\star = Y \).
  Above, we have proved that \( \mu(t) \rightarrow \omega \) in \( \mathcal{P}_2(\mathbb{R}) \).
  It is easily seen from \eqref{eq:evol} that the support of $\mu(t)$ remains bounded uniformly in time. 
  Consequently, convergence holds in \emph{any} Wasserstein distance $W_p$ with finite $p$ as well.
  However, we cannot in general expect convergence in $W_\infty$. 
  Indeed, assume that \( \underline X:=\inf \supp \mu_0 < \inf \supp \omega \).
  Then the left edge $X(t,0)$ of $\mu(t)$'s support satisfies
  \begin{equation*}
    \partial_t X(t,0) = 2G(X(t,0)), \quad X(0,0) = \underline X.
  \end{equation*}
  Since $G(\underline X)=0$,
  this initial value problem has the (unique, since $G$ is Lipschitz continuous) solution $X(t,0)=\underline X$ for all $t\ge0$.
  Hence the left edge of $\mu(t)$'s support remains at constant distance to $\omega$'s support.
\end{remark}
\begin{theorem}[Asymptotics for smaller mass]
  \label{thm:q1bis}
  Assume that $q_a=q_r=1$ and $m<1$.
  Let $\mu:[0,\infty)\to\prb(\setR)$ be a transient solution in the sense of Definition \ref{dfn:transient}.
  Then $\mu$ converges to $\omega$ vaguely as $t\to\infty$.
\end{theorem}
\begin{proof}
  First, we are going to prove that $X(t;z)\to-\infty$ as $t\to\infty$ for all $z<(1-m)/2$.
  Indeed, if $z<(1-m)/2$, then \eqref{eq:evol} in combination with $m -2 G(X(t,z)) \leq m$ 
  implies that
  \begin{align*}
    \partial_t X(t,z) = 2z - 1 + m - 2 G(X(t,z)) \leq 2z - 1 + m <0,
  \end{align*}
  hence $X(t,z)\le X_0(z)+(2z-1+m)t\to-\infty$.
  An analogous argument proves that $X(t;z)\to+\infty$ for all $z>(1+m)/2$.

  Next, observe that the stationary equation \eqref{eq:360} is satisfied 
  by $\widetilde{X}^\star(z) := Y(z - \frac{1-m}{2})$ for all \( z \in \left[ (1-m)/2, (1+m)/2 \right] \),
  i.e., for those $z$ we have $2z-1=G(\widetilde X^\star(z))$.
  By applying now the same arguments as in the proof of Theorem \ref{thm:q1}
  --- restricting $z$ to the interval $\left[ (1-m)/2, (1+m)/2 \right]$ ---
  one shows that $X(t,z)\to\widetilde X^\star(z)$ monotonically for all feasible $z$.
  We conclude that \( \mu(t) \) converges vaguely to \( \mu^\star=\omega \): 
  For given \( f \in C^0_c(\mathbb{R}) \), 
  we have
  \begin{align*}
    \int_{\mathbb{R}} f(x) \diff \mu(t,x) = {} & \int_{0}^{1} f(X(t,z)) \diff z\\
    = {} & \int_{(1-m )/2}^{1-(1-m)/2} f(X(t,z)) \diff z \nonumber \\
    {} & + \int_{0}^{(1-m)/2} f(X(t,z)) \diff z + \int_{1-(1-m)/2}^1 f(X(t,z)) \diff z\\
    \rightarrow {} & \int_{(1-m )/2}^{1-(1-m)/2} f(\widetilde X^\star(z)) \diff z = \int_{\mathbb{R}} f(x) \diff \omega(x), \quad t\rightarrow\infty,
  \end{align*}
  by the substitution formula (Lemma \ref{lem:16}) and the dominated convergence theorem.
\end{proof}

\subsection{Equilibration in the general case}
\label{sec:conv-against-steady}
In Sections \ref{sec:case-q-2} and \ref{sec:case-q-1} above, 
we clarified the asymptotic behavior of solutions in the cases \( q_a = q_r = 2 \) and \( q_r = 1 \leq q_a \le 2\) in detail, 
also in situations of shortage and excess of mass for the datum $\omega$.

In this section we establish convergence to a steady state under the more general conditions
\begin{align}
  \label{eq:362}
  1&\le q_r < q_a \leq 2, \quad \text{or}\\
  1&\le q_r = q_a \le 4/3,
\end{align}
and under the additional hypothesis that $m=\omega(\setR)=1$.
Our method of proof uses an energy-energy-dissipation inequality 
combined with a moment bound for the sublevels of the energy, see Proposition \ref{prp:moment-bound} below, 
in order to derive both compactness of trajectories and continuity of the dissipation. 


We start by summarizing some tools from measure theory and the variational properties of the energy functional we need.
\begin{definition}[Uniform integrability]
  A measurable function \( f : \R \rightarrow [0,\infty] \) is \emph{uniformly integrable} 
  with respect to a sequence $(\mu_n)_{n\in\setN}$ of finite measures 
  if
  \begin{equation}
    \label{eq:46}
    \lim_{M \rightarrow \infty} \sup_n \int_{\left\{ f(x) \geq M\right\}} f(x) \diff \mu_n(x) = 0.
  \end{equation}
\end{definition}
\begin{lemma}[Continuity of integral functionals]
  \label{lem:4}
  \cite[Lemma 5.1.7]{AGS08}
  Let $(\mu_n)_{n \in \mathbb N}$ be a sequence in  $\mathcal{P}(\mathbb{R})$  converging narrowly to $\mu \in \mathcal{P}(\setR)$. 
  If $f:\setR\to\setR$ is a continuous function that is uniformly integrable \wrt $(\mu_n)_{n\in\setN}$, 
  then
  \begin{align}
    \label{eq:52}
    \lim_{n\rightarrow\infty} \int_{\mathbb{R}} f(x) \diff \mu_n(x) = {} & \int_{\mathbb{R}} f(x) \diff \mu(x).
  \end{align}
\end{lemma}
\begin{lemma}[Uniform integrability of moments]
  \label{lem:24}
  \cite[Corollary to Theorem 25.12]{95-Billingsley-Proba_and_Measuer}
  A sequence $(\mu_n)_{n\in\setN}$ of probability measures satisfying
  \begin{equation}
    \label{eq:316}
    \sup_n \int_{\mathbb{R}} \left| x \right|^r \diff \mu_n(x) < \infty
  \end{equation}
  for some $r>0$ is tight.
  Moreover, each of the functions $x \mapsto \left| x \right|^q$ with $0<q<r$ is uniformly integrable \wrt $(\mu_n)_{n\in\setN}$.
\end{lemma}
\begin{lemma}[Convergence of product measures]
  \label{lem:27}
 From \cite[Theorem 2.8]{68-Billingsley-Conv-Proba} it follows that if \( \left( \mu_n \right)_n \), \( \left( \nu_n \right)_n \) are two sequences in \( \mathcal{P}(\R) \) and \( \mu, \nu \in \mathcal{P}(\R) \), then
  \begin{equation*}
    \mu_n \otimes \nu_n \rightarrow \mu \otimes \nu \text{ narrowly} \Leftrightarrow \mu_n \rightarrow \mu \text{ and } \nu_n \rightarrow \nu \text{ narrowly}.
  \end{equation*}
\end{lemma}
For what follows, we recall the definition of the energy functional:
\begin{align}
  \label{eq:nrg}
  \mathcal{E}[\mu] 
  = - \frac{1}{2} \int_{\mathbb R}\int_{\mathbb R} \psi_r(x-y)  \diff \mu(x) \diff \mu(y)  
  + \int_{\mathbb R} \int_{\mathbb R} \psi_a(x-y) \diff \omega(x) \diff \mu(y).
\end{align}
\begin{proposition}[Moment bound]
  \label{prp:moment-bound}
  \cite[Theorem 2.7, Theorem 4.1]{13-FornasierHuetter-VarProperties}
  Under our general hypotheses that $1\le q_r\le q_a\le2$ and $\omega \in \mathcal{P}_2(\mathbb{R})$,
  the energy functional $\mathcal{E}$ is bounded from below in $\mathcal{P}_2(\mathbb{R})$.
  Moreover,
  \begin{itemize}
  \item if $q_a > q_r$, then the \( q_a\)th moment, and 
  \item if $q_a = q_r$, then the \( r \)th moments for all \( 0 < r < q_a/2 \) 
  \end{itemize}
  are uniformly bounded in each of the sublevels of $\mathcal{E}$ in $\mathcal{P}_2(\mathbb{R})$.
\end{proposition}
As mentioned, this result was proven in \cite{13-FornasierHuetter-VarProperties}. 
To facilitate the reading, the main ideas of the proof are briefly revisited in Appendix \ref{sec:moment-bound-subl}. 
%
%
\begin{lemma}[Dissipation formula]
  \label{lem:28}
  Let $1<q_r\le q_a\le2$, and let $\mu:[0,\infty)\to\prb_2(\setR)$ be a transient solution in the sense of Definition \ref{dfn:transient}. 
  Then $t\mapsto\mathcal{E}[\mu(t)]$ is differentiable, and its dissipation is given by
  \begin{align}
    \label{eq:368}
    - \frac{\diff}{\diff t} \mathcal{E}[\mu(t)] 
    = \mathcal{D}[\mu(t)]
    := \int_{0}^{1} \left| \int_{0}^{1} \left[ \psi_a'(X(t,z) - Y(\zeta)) - \psi_r'(X(t,z) - X(t,\zeta)) \right] \diff \zeta \right|^2 \diff z.
  \end{align}
\end{lemma}
\begin{proof}
  By Lemma \ref{lem:16}, the energy can be written in the form
  \begin{align*}
    \mathcal{E}[\mu(t)] 
    = \int_{0}^{1} \int_{0}^{1} \psi_a(X(t,z) - Y(\zeta)) \diff \zeta \diff z 
    - \frac{1}{2} \int_{0}^{1} \int_{0}^{1} \psi_r(X(t,z) - X(t, \zeta)) \diff \zeta \diff z.
  \end{align*}
  Since \( \psi_r' \) and \( \psi_a' \) are continuous functions and \( X(.,.) \in C^1([0,T], \cadlag[0,1]) \), the appearing derivatives will be bounded uniformly in \( t \in [0,T] \), so differentiating under the integral sign is justified by the dominated convergence theorem, yielding
  \begin{align}
    \frac{\diff}{\diff t} \mathcal{E}[\mu(t)] = {} & \int_{0}^{1} \int_{0}^{1} \psi_a'(X(t,z) - Y(\zeta)) \partial_t X(t,z) \diff \zeta \diff z \nonumber \\
    & \quad - \frac{1}{2} \int_{0}^{1} \int_{0}^{1} \psi_r'(X(t,z) - X(t, \zeta)) \left( \partial_t X(t,z) - \partial_t X(t,\zeta) \right) \diff \zeta \diff z \nonumber \\
     {} = & \int_{0}^{1} \int_{0}^{1} \int_{0}^{1} \left[ \psi_a'(X(t,z) - Y(\zeta)) - \psi_r'(X(t,z) - X(t, \zeta)) \right] \nonumber\\
     & \quad \cdot \left[ - \psi_a'(X(t,z) - Y(\xi)) + \psi_r'(X(t,z) - X(t,\xi)) \right] \diff \xi \diff \zeta \diff z \label{eq:343}\\
     = {} & -\int_{0}^{1} \left| \int_{0}^{1} \left[ \psi_a'(X(t,z) - Y(\zeta)) - \psi_r'(X(t,z) - X(t,\zeta)) \right] \diff \zeta \right|^2 \diff z,\nonumber
  \end{align}
  where in equation \eqref{eq:343}, we inserted \eqref{eq:evol}, and used the anti-symmetry of \( \psi_r' \).
\end{proof}
We are now in the position to state and prove the main result of this section.
\begin{theorem}[Convergence of a sequence]
  \label{thm:1q2}
  Assume that either
  \begin{equation}
    \label{eq:346}
    1 < q_r < q_a < 2
  \end{equation}
  or 
  \begin{equation}
    \label{eq:347}
    1 < q_a = q_r < \frac{4}{3}.
  \end{equation}
  Then, for each transient solution $\mu:[0,\infty)\to\prb(\setR)$, 
  there is a sequence \( (t_k)_{k \in \mathbb{N}} \) with $t_k \rightarrow \infty$ 
  such that $\mu(t_k)$ converges narrowly to a steady state $\mu^\star$.
\end{theorem}
\begin{proof}
  By Lemma \ref{lem:28}, we know that
  \begin{align}
    \label{eq:dissip}
    \int_{0}^{t} \mathcal{D}[\mu(\tau)] \diff \tau =
    \mathcal{E}[\mu(0)] - \mathcal{E}[\mu(t)] ,
  \end{align}
  where \( \mathcal{D}[\mu(\tau)] \geq 0 \) for all \( \tau \), so the energy is monotonically decreasing.
  Moreover, by Proposition \ref{prp:moment-bound}, we know that \( \mathcal{E} \) is bounded from below, 
  yielding convergence of the integral in \eqref{eq:dissip} for \( t\rightarrow\infty \).
  Therefore, there is a sequence of times $t_k\to\infty$ for which
  \begin{equation*}
    \mathcal{D}[\mu(t_k)] \rightarrow 0, \quad\text{ as $k \rightarrow\infty$}.
  \end{equation*}
  Recalling our hypotheses \eqref{eq:346} and \eqref{eq:347} on $q_a$ and $q_r$,
  it follows by the second part of Proposition \ref{prp:moment-bound}
  that some positive moment of $\mu(t_k)$ is uniformly bounded as $t_k\to\infty$.
  By Lemma \ref{lem:24}, there is a subsequence --- still denoted by $(t_k)_{k\in\setN}$ --- 
  such that $\mu(t_k)$ converges narrowly to a limit $\mu^\star\in \mathcal{P}(\mathbb{R})$.

  It remains to verify that $\mu^\star$ is a stationary state.
  This is done by proving that $\mathcal{D}[\mu^\star] = 0$ as follows.
  First, expand \( \mathcal{D} \) into a sum of triple integrals \wrt probability measures:
  \begin{align}
    \label{eq:352}
    \begin{split}
      \mathcal{D}[\mu] 
      & = \int_{\mathbb{R}^3} \psi_a'(x-y') \psi_a'(x-z') \diff \mu(x) \diff \omega (y') \diff \omega(z') \\
      & + \int_{\mathbb{R}^3} \psi_r'(x-y) \psi_r'(x-z) \diff \mu(x) \diff \mu (y) \diff \mu(z) \\
      & - 2\int_{\mathbb{R}^3} \psi_r'(x-y) \psi_a'(x-z')  \diff\mu(x)\diff \mu (y) \diff \omega(z').
    \end{split}
  \end{align}
  Next, observe that by Lemma \ref{lem:27}, the narrow convergence of $\mu(t_k)$
  is inherited by the tensorized sequences 
  $\mu(t_k)\otimes\mu(t_k)\otimes\mu(t_k)$, $\mu(t_k)\otimes\mu(t_k)\otimes\omega$ and $\mu(t_k)\otimes\omega\otimes\omega$,
  which converge to 
  $\mu^\star\otimes\mu^\star\otimes\mu^\star$, $\mu^\star\otimes\mu^\star\otimes\omega$ and $\mu^\star\otimes\omega\otimes\omega$, respectively.
  Finally, to conclude that 
  \begin{align*}
    \mathcal{D}[\mu^\star] = \lim_{k\to\infty}\mathcal{D}[\mu(t_k)] = 0
  \end{align*}
  by means of Lemma \ref{lem:4},
  it suffices to verify that the integrands appearing in \eqref{eq:352} are uniformly integrable functions
  with respect to their respective tensorized measures.
  By definition of $\psi_a$, $\psi_r$ and elementary estimates, we have:
  \begin{align*}
    |\psi_a'(x-y')\psi_a'(x-z')| &\le C|x-y'|^{q_a-1}|x-z'|^{q_a-1} \le C\big(1+|x|^{2q_a-2}+|y'|^{q_a-1}+|z'|^{q_a-1}\big), \\
    |\psi_r'(x-y)\psi_r'(x-z)| &\le C|x-y|^{q_r-1}|x-z|^{q_r-1} \le C\big(1+|x|^{2q_r-2}+|y|^{q_r-1}+|z|^{q_r-1}\big), \\
    |\psi_r'(x-y)\psi_a'(x-z')| &\le C|x-y|^{q_r-1}|x-z'|^{q_a-1} \le C\big(1+|x|^{q_a+q_r-2}+|y|^{q_a-1}+|z'|^{q_a-1}\big).
  \end{align*}
  In view of our hypotheses on $q_a$ and $q_r$, and since $\omega$ is of compact support,
  uniform integrability of all integrands is guaranteed if the $\mu(t_k)$ have uniformly bounded moments of some order $r>2q_a-2$.
  We apply Proposition \ref{prp:moment-bound} one more time:
  in the attraction dominated case $q_a>q_r$, 
  we have uniformly bounded moments of order $q_a>2q_a-2$, for all $q_a<2$,
  which finishes the proof under hypothesis \eqref{eq:346}.
  In the balanced case $q_a=q_r$, all moments of order less than $q_a/2$ are uniformly bounded.
  To finish the proof under hypothesis \eqref{eq:347}, observe that \( 2q_a - 2 < q_a/2 \) if \( q_a < 4/3 \).
\end{proof}
\begin{remark}
  We indicate why our method of proof that is based on uniform moment bounds 
  cannot be simply generalized to a larger set of exponents $q_a$, $q_r$.
  Formally assume that \( \omega = \delta_0 \) --- which can be weakly approximated by $\omega\in L^\infty_c$.
  In the course of the proof, we need to verify that $(x,y',z')\mapsto\psi_a(x-y')\psi_a(x-z')$ is uniformly integrable
  with respect to the measures $\mu(t_k)\otimes\omega\otimes\omega$.
  This is equivalent to uniform integrability of
  \begin{equation*}
    x\mapsto\int_{\mathbb{R} \times \mathbb{R}} \left| x-y' \right|^{q_a - 1} \left| x-z' \right|^{q_a - 1} \diff \omega(y') \diff \omega(z') 
    = \left| x \right|^{2q_a - 2}
  \end{equation*}
  with respect to the measures $\mu(t_k)$.
  Hence, we would need a uniform bound on $\mu(t_k)$'s moments of some order $r>2q_a-2$.
\end{remark}

\section{Conclusion}
\label{sec:conclusion2}
%
The reformulation of the evolution equation \eqref{eq:501} in terms of the pseudo-inverse function 
proved very helpful for the analysis of the asymptotic behavior in the cases $q_r=1\le q_a\le2$ and $q_r=2=q_a$.
In the new variables, the equation becomes local.
This propery is lost in the more general situation $1<q_r\le q_a<2$.
Here, we exploted the underlying gradient flow structure and in particular a coercivity property the of the energy functional, 
provided by the results in Appendix \ref{sec:moment-bound-subl}.

There is a large range of parameters for $q_a$ and $q_r$ in which the analysis of the asymptotic behavior remains still open, 
in particular the characterization of the steady states, 
which is likely to necessitate additional or completely different techniques compared to the ones used here.

\appendix
\section{Proof of Theorem \ref{thm:existence}}
\label{sec:proof-theor-refthm:l}
The following is a crucial technical ingredient for the proof of Theorem \ref{thm:existence}.
\begin{lemma}
  \label{lem:18}
  Let \(\omega \in L^\infty(\mathbb{R}) \cap L^1(\mathbb{R})\) such that \( \omega \geq 0 \). 
  Then \(U=\psi_a' \ast \omega\) is Lipschitz-continuous.
\end{lemma}
\begin{proof}
  For \(q_a = 1\), remember that \(\psi'(x) = \sgn(x)\) by definition, leading to (see \eqref{eq:328})
  \begin{equation*}
    \psi_a' \ast \omega (x) = 2 \int_{-\infty}^{x} \omega(y) \, \mathrm{d} y - \left\| \omega \right\|_{1},
  \end{equation*}
  which is obviously Lipschitz-continuous if \(\omega \in L^{\infty}(\mathbb{R})\).

  For \(q_a \in (1,2)\), we have $\psi_a''(x) =q_a(q_a-1)|x|^{q_a-2}$, 
  which is integrable on $[-1,1]$ and bounded by $q_a(q_a-1)$ on $\mathbb R \setminus [-1,1]$.
  Hence
  \begin{align*}
    \left| \psi_a'' \ast \omega(x) \right| = {} & \int_{-1}^{1} q_a(q_a-1) \left| y \right|^{q_a-2} \omega (x-y) \diff y\\
    & + q_a(q_a-1) \int_{\mathbb{R} \setminus [-1,1]} \left| y \right|^{q_a-2} \omega(x-y) \diff y\\
    \leq {} &  q_a(q_a-1) \left( \frac{2}{q_a-1} \left\| \omega \right\|_\infty + \left\| \omega \right\|_1 \right),
  \end{align*}
  which means that $\psi_a''\ast\omega$ is bounded.
  Therefore, $\psi_a'\ast\omega$ is Lipschitz continuous.
\end{proof}
To prove Theorem \ref{thm:existence} we follow the classical strategy from semigroup theory:
we define an operator whose fixed point corresponds to a transient solution in the sense of Definition \ref{dfn:transient},
and then we apply the Banach fixed point theorem to conclude both existence and uniqueness.
A similar approach has been used for a related evolution equation in \cite[Theorem 2.9]{BDi08},
however:
\begin{itemize}
\item the lack of Lipschitz-continuity of $\psi_a'$ and $\psi_r'$ forces us to define the fixed-point operator in a more involved way,
  using a time rescaling, in order to guarantee its contractivity;
\item our way to verify the distributional formulation \eqref{eq:weak} is much more direct. 
\end{itemize}
\begin{proof}[Proof of Theorem \ref{thm:existence}]
  For now, we assume \(q_r \in (1,2]\). 
  The (simpler) case \(q_r = 1\) is discussed afterwards in Step 6.

  In the following, fix some \( \alpha > 0 \) such that
  \begin{equation*}
    \omega(x) \leq \alpha^{-1}, \ \mu_0(x) \leq \alpha^{-1}, \ \text{for a.e.\@ } x \in \mathbb{R}.
  \end{equation*}
  \begin{enumerate}
  \item[Step 1.] \emph{(Definition of the operator)} 
    Let \(T > 0\) be a fixed time horizont. 
    By Lemma \ref{lem:18}, the function \(U=\psi_a' \ast \omega\) is Lipschitz-continuous. 
    Denote its Lipschitz-constant by $\lambda$ and set
    \begin{equation*}
      U_\lambda(x) := U(x) - \lambda x, \quad x \in \mathbb R.
    \end{equation*}
    Define the fixed point operator
    \begin{eqnarray}
      S[X](t,z)&:=& \exp(-\lambda t)X_0(z) \nonumber \\
      &+&\int_0^t \exp(-\lambda(t-s)) \left[ \int_0^1 \psi_r'(X(s,z) - X(s,\zeta)) \, \mathrm{d} \zeta - U_\lambda(X(s,z)) \right] \, \mathrm{d} s, \label{eq:256}
    \end{eqnarray}
    on the set
    \begin{equation*}
      \mathcal{B} := \left\{ X(.,.) \in C\big([0,T], \cadlag([0,1))\big) :
        \text{$X$ fulfills \eqref{eq:258}}
      \right\},
    \end{equation*}
    where the slope condition \eqref{eq:258} is given by
    \begin{equation}
      \label{eq:258}
      \begin{gathered}
        \frac{1}{h} \left(X(t,z+h) - X(t,z) \right) \geq \alpha\exp(-\lambda t) \\
        \text{for all } h \in (0,1) \text{ and } z \in [0,1-h),
      \end{gathered}  \tag{SL}.
    \end{equation}
    We endow $\mathcal B$ with the norm 
    \begin{equation*}
      \lVert X \rVert_{\mathcal{B}} := \sup \left\{ \exp(\lambda t) |X(t,z)| : t \in [0,T],\,z\in[0,1)\right\}.
    \end{equation*}
    By construction, \(U_\lambda\) is nonincreasing. 
    Note that $\mathcal B$ is a closed subset of the Banach space $C([0,T],\cadlag([0,1)))$,
    since convergence in the $\mathcal B$-norm is obviously equivalent to uniform convergence on $[0,T]\times[0,1)$,
    and the slope condition \eqref{eq:258} passes to the pointwise limit.
  \item[Step 2.] \emph{($S$ maps $\mathcal{B}$ into $\mathcal{B}$)} \label{item:s-b-into-b} 
    First, for \(X \in \mathcal{B}\), the continuity of \(t \mapsto S[X](t,.)\) from \([0,T]\) to \(\cadlag([0,1))\) follows 
    from the continuity of the integral defining \(S\) and by continuity of the functions involved.
    
    Second, we verify that $S$ propagates the slope condition \eqref{eq:258}.
    Let $X \in \mathcal{B}$ (in particular non-decreasing), $h > 0$ and $z \in [0, 1 - h)$ be given. 
    Using that $X_0$ satisfies the slope condition \eqref{eq:258} by our choice of $\alpha$, 
    and that both $\psi_r'$ and $U_\lambda$ are monotone,
    we obtain
    \begin{equation}
      \frac{1}{h}\left[ S[X](t,z + h) - S[X](t,z) \right] \geq \frac{\e^{-\lambda t}}h\big[X_0(z+h)-X_0(z)]
      \ge\alpha\e^{-\lambda t}.\label{eq:261}
    \end{equation}
  \item[Step 3.] \emph{($S$ is contractive)} \label{item:linfty-contractive} Let
    $X,\widetilde{X} \in \mathcal{B}$. Then,
    \begin{eqnarray}
      && \exp(\lambda t) \cdot \left| S[\widetilde{X}](t,z) - S[X](t,z)\right| \nonumber\\
      &\leq & \int_0^t \exp(\lambda s) \int_0^1 \left| \psi_r'(\widetilde{X}(s,z) - \widetilde{X}(s,\zeta)) - \psi_r'(X(s,z) - X(s,\zeta))\right| \, \mathrm{d} \zeta \, \mathrm{d} s\label{eq:262}\\
      &+& \int_0^t \exp(\lambda s) \left| U_\lambda(\widetilde{X}(s,z)) - U_\lambda(X(s,z))\right| \, \mathrm{d} s. \label{eq:263}
    \end{eqnarray}
    We can bound \eqref{eq:263} by using the Lipschitz-continuity of $U_\lambda$ with Lipschitz-constant \(2\lambda\) and derive the estimate
    \begin{align*}
      \int_0^t \exp(\lambda s) \left| U_\lambda(\widetilde{X}(s,z)) -
        U_\lambda(X(s,z))\right| \, \mathrm{d} s \leq 2\lambda \int_0^t \lVert \widetilde{X} - X
      \rVert_{\mathcal{B}} \, \mathrm{d} s \leq 2\lambda t \lVert \widetilde{X} - X
      \rVert_{\mathcal{B}}.
    \end{align*}
    Boundedness of the integral in \eqref{eq:262} is more difficult to obtain, since $\psi_r'$ is \emph{not} Lipschitz-continuous. 
    Assume that $\zeta \leq z$ and without loss of generality that
    \begin{equation}
      \label{eq:667}
      \widetilde{X}(s,z) - \widetilde{X}(s,\zeta) \geq X(s,z) - X(s,\zeta).
    \end{equation}
    Both differences are non-negative.
    The slope condition \eqref{eq:258}, even provides a lower bound:
    \begin{align*}
      \widetilde{X}(s,z) - \widetilde{X}(s,\zeta) &\geq \alpha\e^{-\lambda t}(z - \zeta), \quad
      X(s,z) - X(s,\zeta) \geq \alpha\e^{-\lambda t}(z - \zeta).
    \end{align*}
    Thus $\psi_r'$ is applied to non-negative arguments whose difference is estimated as follows:
    \begin{equation*}
      0\le\big[\widetilde{X}(s,z) - \widetilde{X}(s,\zeta)\big] - \big[X(s,z) - X(s,\zeta)\big]
      \leq 2 \sup_{\zeta \in [0,1]} \left| X(s,\zeta) - \widetilde{X}(s,\zeta)\right|
      \le 2 e^{-\lambda s}\|\widetilde{X}-X\|_{\mathcal B}.
    \end{equation*}
    Since $\psi_r'$ is monotonically increasing and concave for positive arguments,
    it is easily seen that
    \begin{align*}
      0&\le\psi_r'\big(\widetilde X(s,z)-\widetilde X(s,\zeta)\big) - \psi_r'\big(X(s,z)-X(s,\zeta)\big)\\
      &\le\psi_r'\big(\alpha\e^{-\lambda s}\big[(z-\zeta)+2\|\widetilde{X}-X\|_{\mathcal B}\big]\big)
      -\psi_r'\big(\alpha\e^{-\lambda s}(z-\zeta)\big)
    \end{align*}
    A completely analogous reasoning applies if the inequality \eqref{eq:667} is inverted, or if $\zeta\ge z$ instead.
    The mean value theorem, applied to the function $\psi_r'$ with derivative $\psi_r''(x)=q_r(q_r-1)$, 
    leads to the estimate
    \begin{align*}
      & \left| \psi_r'(\widetilde{X}(s,z) - \widetilde{X}(s,\zeta)) - \psi_r'(X(s,z) - X(s,\zeta))\right| \\
      &\le 2q_r(q_r - 1) \alpha^{q_r - 2}\e^{-\lambda(q_r-1)s}|z-\zeta|^{q_r - 2}\|\widetilde{X}-X\|_{\mathcal B},
    \end{align*}
    which holds uniformly in $s\in[0,T]$ and $z,\zeta\in[0,1)$.
    Integration with respect to $\zeta$ yields
    \begin{align*}
      &\int_0^1 \left| \psi_r'(\widetilde{X}(s,z) - \widetilde{X}(s,\zeta)) - \psi_r'(X(s,z) - X(s,\zeta))\right|\, \mathrm{d} \zeta \\
      &\leq 2q_r(q_r-1) \alpha^{q_r - 2}\e^{-\lambda(q_r-1)s} \bigg(\int_0^1\left|z-\zeta\right|^{q_r-2}\, \mathrm{d} \zeta\bigg) \|\widetilde{X}-X\|_{\mathcal B},\\
      &\leq K \e^{-\lambda(q_r-1)s} \|\widetilde{X}-X\|_{\mathcal B},
    \end{align*}
    where $K$ is a finite constant that depends on $q$ and $\alpha$ only.
    Notice that our assumption $q_r>1$ is important for the finiteness of the integral.
    Altogether, we conclude that
    \begin{align*}
      &\int_0^t \e^{\lambda s} \int_0^1 \left| \psi_r'(\widetilde{X}(s,z) -
        \widetilde{X}(s,\zeta)) - \psi_r'(X(s,z) - X(s,\zeta))\right| \, \mathrm{d} \zeta \, \mathrm{d} s \\
      &\leq K\bigg(\int_0^Te^{\lambda(2-q_r)s}\dd s\bigg) \lVert X - \widetilde{X} \rVert_{\mathcal{B}}.
    \end{align*}
    Thus, for a sufficiently small choice of $T>0$, the operator $S$ is indeed a contraction.

    Combining the previous steps, we find a unique fixed point $X$ of $S$ using the Banach fixed point theorem, i.e.,\@ an \(X \in \mathcal{B}\) such that
    \begin{eqnarray}
    X(t,z) &=& \exp(-\lambda t)X_0(z) \nonumber \\
      & +&\exp(-\lambda t)\int_0^t \underbrace{\exp(\lambda s) \left[ \int_0^1 \psi'_r(X(s,z) - X(s,\zeta)) \, \mathrm{d} \zeta - U_\lambda(X(s,z)) \right]}_{\text{integrand}} \, \mathrm{d} s, \label{eq:271}
    \end{eqnarray}
    where the integrand is continuous as a mapping from \([0,T]\) to \(\cadlag([0,1])\), again by the continuity of the involved functions and the boundedness of $X$. 
    Hence, the right-hand side has the desired \(C^{1}\)-regularity on \([0,T]\) and so has \(X\),  by the equality in \eqref{eq:271}.
  \item[Step 4.] \emph{(Global existence)}\label{item:linfty-global-existence} Differentiating \eqref{eq:271} with respect to time directly yields
    \begin{equation}
      \partial_tX(t,z) = \int_0^1 \psi_r'(X(t,z) - X(t,\zeta)) \, \mathrm{d} \zeta - U(X(t,z)),\label{eq:272}
    \end{equation}
    hence $X$ fulfills also the desired equation \eqref{eq:evol} to define a transient solution. 
    To conclude global existence, it suffices to verify that the $L^\infty$-norm of $X$ does not explode at finite time:
    by Lipschitz-continuity of $U$, the estimate $\left| \psi_r'(x)\right| \leq q_r \cdot (1 + \left| x\right|)$, and by Gronwall's inequality, 
    we  obtain
    \begin{equation*}
      \lVert X(t,.) \rVert_{L^\infty} \leq \left(\lVert X_0 \rVert_{L^\infty} + C_{1}t\right) \, \exp(C_{2}t).
    \end{equation*}
  \item[Step 5.] \label{item:linfty-dist-form} \emph{(Distributional formulation)} 
    First, for every \(t \in [0,\infty)\), \(X(t,.)\) is a right continuous increasing function 
    and hence defines a probability measure $\mu(t,x)$ on \(\mathbb{R}\) with $X(t,\cdot) = X_{\mu(t,\cdot)}(\cdot)$. 

    Second, let \(\varphi \in C^{\infty}_{c}([0,\infty)\times \mathbb{R})\). 
    As we have \(C^{1}\)-regularity of the solution curve $t \mapsto X(t,z)$, 
    combining this with the fundamental theorem of calculus, Fubini’s theorem and the compactness of the support of \(\varphi\), 
    we arrive that
    \begin{align}
      \int_{0}^{\infty} \int_{0}^{1} \frac{\mathrm{d}}{\mathrm{d} t} \left[\varphi(t, X(t,z))\right] \, \mathrm{d} z \, \mathrm{d} t {} & =
      -\int_{0}^{1}
      \varphi(0,X(0,z)) \, \mathrm{d} z && \nonumber \\
      & = -\int_{\mathbb{R}} \varphi(0,x) \, \mathrm{d} \mu(0,x), \label{eq:x275}
    \end{align}
    where we used Lemma \ref{lem:16} in the inequality, since \(\varphi(0,.)\) is bounded and therefore in \(L^{1}(\mu(0))\).

    On the other hand, again by the regularity of the curves and the chain rule, for all $t \in [0,\infty)$ and almost all $z \in [0,1]$,
    \begin{equation}
      \label{eq:275}
      \frac{\mathrm{d}}{\mathrm{d} t} \left[\varphi(t, X(t,z))\right] = \partial_{t}\varphi(t,X(t,z)) +
      \partial_{x}\varphi(t,X(t,z)) \cdot \partial_{t}X(t,z).
    \end{equation}
    The integration of the first term in \eqref{eq:275} yields
    \begin{equation}\label{eq:y275}
      \int_{0}^{\infty} \int_{0}^{1} \partial_{t}\varphi(t,X(t,z)) \, \mathrm{d} z \, \mathrm{d} t = \int_{0}^{\infty} \int_{\mathbb{R}}
      \partial_{t}\varphi(t,x) \, \mathrm{d} \mu(t,x) \, \mathrm{d} t,
    \end{equation}
    where we again used Lemma \ref{lem:16} as above. By inserting equation \eqref{eq:272} for $\partial_{t}X$, the integration of  the second term in \eqref{eq:275} becomes
      \begin{eqnarray}
    & &\int_{0}^{\infty} \int_{0}^{1} \partial_{x}\varphi(t,X(t,z)) \cdot \partial_{t}X(t,z) \,\mathrm{d}z \,\mathrm{d} t \nonumber \\
     &=&\int_{0}^{\infty} \int_{0}^{1} \partial_{x}\varphi(t,X(t,z)) \int_{0}^{1} \psi_r'(X(t,z) - X(t,\zeta)) \, \mathrm{d} \zeta - U(X(t,z) \, \mathrm{d} z \, \mathrm{d} t \nonumber \\
        &= & \int_{0}^{\infty} \int_{0}^{1} \partial_{x}\varphi(t,X(t,z))\Big[(\psi_r' \ast
        \mu(t,.))(X(t,z)) - U(X(t,z)) \Big] \, \mathrm{d} z \, \mathrm{d} t \nonumber \\
        & = & \int_{0}^{\infty} \int_{\mathbb{R}} \partial_{x}\varphi(t,x) \cdot (\psi_r' \ast \mu(t,.))(x) \, \mathrm{d}
        \mu(t,x) \, \mathrm{d} t \nonumber \\
         & -& \int_{0}^{\infty} \int_{\mathbb{R}} \partial_{x}\varphi(t,x) \cdot U(x) \, \mathrm{d}
        \mu(t,x) \, \mathrm{d} t,\label{eq:277}.
       \end{eqnarray}
      The use of Lemma \ref{lem:16} here is justified because the involved measures are compactly supported, yielding a bound on their second moment; this results in \(\psi_r' \ast \mu(t) \in L^{1}(\mu(t))\) and \(U \in L^{1}(\mu(t))\), which we then combine with \(\partial_{x}\varphi(t,.) \in L^{\infty}(\mathbb{R})\) to see that the integrand in the last line of \eqref{eq:277} is in \(L^{1}(\mu(t))\).  Then \eqref{eq:277} together with \eqref{eq:x275}, \eqref{eq:275}, and  \eqref{eq:y275} lead to the desired equation \eqref{eq:weak}.

    \item[Step 6.] \label{item:linfty-changes-q-1} \emph{(Adjustments for $q_r=1$)} We first derive formally  the simplified pseudo-inverse equation \eqref{eq:229}. Note that for a strictly increasing pseudo-inverse \(X(t,.)\) with associated measure \(\mu(t)\) and CDF \(F(t,.)\), \(X(t,.)\) is the right-inverse of \(F(t,.)\), which means that we can write \eqref{eq:229} as
      \begin{equation}
        \label{eq:278}
        \partial_{t}X(t,z) = 2 F(t,X(t,z)) - 1 - U(X(t,z)) = 2 z - 1 - U(X(t,z)).
      \end{equation}
      We can now apply again the previous arguments to find a solution to this equation and afterwards justify that \(X(t,.)\) stays indeed strictly increasing, allowing us to follow the above equation \eqref{eq:278} in reverse direction.

      As already mentioned, being \(\omega\) assumed to be absolutely-continuous with its density belonging to \(L^{\infty}(\mathbb{R})\) implies that the attraction potential \(U(.)\) is Lipschitz-continuous. Therefore, again denoting its Lipschitz-constant by \(\lambda\), we can define the operator \(S\) analogously to \eqref{eq:256} using the simplified form of \eqref{eq:278} as the right-hand side, i.e.\@
      \begin{align}
        S[X](t,z) := {} &\exp(-\lambda t)X_0(z) \nonumber \\
        &+ \int_0^t \exp(-\lambda(t-s)) \left[ 2z - 1 - \Big(U(X(s,z)) - \lambda X(s,z)\Big)\right] \, \mathrm{d} s. \label{eq:279}
      \end{align}
      Step 2 can again be applied as the integrand in \eqref{eq:279} is continuous and the monotonicity arguments used in \eqref{eq:261} remain true, as well. Now, Step 3  is actually much easier in case of $q_r=1$, since the mapping
      \begin{equation*}
        X \mapsto 2z - 1 - \left(U(X) - \lambda X\right), \ X \in \mathbb{R}
      \end{equation*}
      is obviously Lipschitz-continuous in \(X\).  Hence, we can deduce a fixed point equation defining  \(X(t,.)\) for all \(t\in[0,T]\) and this provides us also with the strict monotonicity of \(X(t,.)\) for all \(t \in [0,T]\), so we can reverse the simplified equation \eqref{eq:278} as intended, for the given interval of time. Finally, Step 4, giving global existence, and Step 5, defining the distributional solution of \eqref{eq:weak}, work analogously and can be followed verbatim. \qedhere
    \end{enumerate}
\end{proof}

\section{Moment bound for the sublevels of the functional}
\label{sec:moment-bound-subl}

For the asymptotic convergence argument in Section \ref{sec:conv-against-steady}, we needed that the sublevels of the functional \( \mathcal{E} \) exhibit  certain uniform moment bounds. These follow from a more in-depth analysis of the variational properties of this functional developed in \cite{13-FornasierHuetter-VarProperties}. Here, for the sake of completeness, we  only sketch the required arguments. Note that in \cite{13-FornasierHuetter-VarProperties} the results hold for probability measures on $\mathbb R^d$ for any $d \geq 1$ and we report them below in such a generality. 

First, let us deal with the case of the attractive power being larger than the repulsive one, \( q_a, q_r \in [1,2] \), \( q_a > q_r \). 
In this case, one can easily prove the moment bound by means of relatively elementary estimates.
\begin{theorem}
  \label{thm:exist-min-strong}
  Let \( q_a, q_r \in [1,2] \) and \( q_a > q_r \). 
  If \( \omega \in \mathcal{P}_2(\mathbb R^d) \), then \( \mathcal{E} \) is bounded from below 
  and the sub-levels of \( \mathcal{E} \) defined as in \eqref{contenergy} have uniformly bounded \( q_a \)th moments in $\mathcal{P}_{q_r}(\mathbb R^d)$.
\end{theorem}
\begin{proof}
  From convexity of the power function $\xi\mapsto\xi^q$ for the relevant exponents $q\in[1,2]$,
  it is easily seen that
  \begin{align}
    \label{eq:16}
    \left| x + y \right|^q &\le 2 \left( \left| x \right|^q + \left| y \right|^q \right),\\
    \left| x - y \right|^q &\ge \left( \frac12 \left| x \right|^q - \left| y \right|^q \right). \label{eq:17}
  \end{align}
  for all \( x,y \in \mathbb{R}^d \).
  Now, let \( \mu \in \mathcal{P}_{q_r}(\mathbb R^d) \). 
  By \eqref{eq:17}, we have
  \begin{align*}
    \int_{\mathbb R^d \times \mathbb R^d} \left| x - y \right|^{q_a} \diff \mu(x) \diff \omega(y) 
    &\geq \int_{\mathbb R^d \times \mathbb R^d} \left(\frac12 \left| x \right|^{q_a} - \left| y \right|^{q_a} \right) \diff \mu(x) \diff \omega(x) \\
    &= \frac12 \int_{\mathbb R^d} \left| x \right|^{q_a} \diff \mu(x) -  \int_{\mathbb R^d} \left| y \right|^{q_a} \diff \omega(y).
  \end{align*}
  On the other hand, by estimate \eqref{eq:16},
  \begin{align*}
    - \frac{1}{2} \int_{\mathbb R^d \times \mathbb R^d} \left| x - y \right|^{q_r} \diff \mu(x) \diff \mu(y) 
    &\geq  - \int_{\mathbb R^d \times \mathbb R^d} \left( \left| x \right|^{q_r} + \left| y \right|^{q_r} \right) \diff \mu(x) \diff \mu(y) \\
    &\geq - 2 \int_{\mathbb R^d} \left| x \right|^{q_r} \diff \mu(x).
  \end{align*}
  In combination, this implies for every radius $R>0$:
  \begin{align*}
    \mathcal{E}[\mu] +  \int_{\mathbb R^d} \left| x \right|^{q_a} \diff \omega(x) 
    &\ge \int_{\mathbb R^d} \left( \frac12\left| x \right|^{q_a} - 2 \left| x \right|^{q_r} \right) \diff \mu(x)\\
    & \ge \frac14\int_{\mathbb R^d}|x|^{q_a}\dd\mu(x) 
    + \int_{\mathbb R^d}\left( \frac14\left| x \right|^{q_a} - 2 \left| x \right|^{q_r} \right) \diff \mu(x)\\
   & \ge \frac14\int_{\mathbb R^d}|x|^{q_a}\dd\mu(x) 
    -2R^{q_r}
    + \int_{\mathbb R^d\setminus B_R(0)}\left( \frac14\left| x \right|^{q_a} - 2 \left| x \right|^{q_r} \right) \diff \mu(x).
 \end{align*}
  Since \( q_a > q_r \), there is an \( R > 0 \) such that the last integral is always non-negative.
  This provides the desired bound on the \( q_a \)th moment.
\end{proof}
Unfortunately, for $q_r > q_a$ the energy functional $\mathcal E$ has in general no minimizers and is not bounded from below, as shown in \cite[Example 2.8]{13-FornasierHuetter-VarProperties}. 
Therefore we address now the limit case of the attractive power being equal to the repulsive one, i.e., $q_r=q_a$, for which we shall be using arguments involving the Fourier transform of the measures \( \mu \) and \( \omega \) to obtain the wished moment bound.

Fix \( q = q_a = q_r \in [1,2) \) and assume \( \omega \in \mathcal{P}_2(\mathbb{R}^d) \). Notice that by completing the squares and setting
\begin{equation*}
  \widetilde{\mathcal{E}}[\mu] := - \frac{1}{2}\int_{\mathbb{R}^d \times \mathbb{R}^d} \psi_q(y - x) \dd [\mu - \omega](x) \dd [\mu - \omega](y),
\end{equation*}
being a quadratic functional of the argument \( \mu - \omega \),  we can write \( \mathcal{E} \) as
\begin{equation*}
  \mathcal{E}[\mu] = \widetilde{\mathcal{E}}[\mu] + C.
\end{equation*}
Quadratic functionals such as $\widetilde{\mathcal{E}}$ can be shown to be non-negative by means of suitable representations in terms of Fourier transforms. 
In fact, assume for a moment that \( f \) and \( g \) were real valued functions in the Schwartz space \( \mathcal{S}(\mathbb{R}^d) \). 
Then, defining the Fourier transform and its inverse, respectively, 
by
\begin{equation*}
  \widehat{f}(\xi) = \int_{\mathbb{R}^d} \exp \left( -i x \cdot \xi \right) f(x) \diff x,\quad
  f^\vee(x) = (2\pi)^{-d}\int_{\mathbb{R}^d} \exp \left( i \xi \cdot x \right) f(\xi) \diff \xi,
\end{equation*}
we can show easily that the following quadratic functional is non-negative:
\begin{align*}
  \int_{\mathbb{R}^d \times \mathbb{R}^d} f(x)\, g(y-x) \, f(y) \diff x \diff y = {} & \int_{\mathbb{R}^d} f \ast g(y) \, f(y) \diff y
  =  \int_{\mathbb{R}^d} f \ast g(y) \, \widehat{f^\vee} (y) \diff y\\
  = {} & \int_{\mathbb{R}^d} \widehat{f}(y) \, \widehat{g}(y) \, f^\vee(y) \diff y
  = \int_{\mathbb{R}^d} | \widehat{f}(y) |^2 \,  \widehat{g}(y) \diff y.
\end{align*}
The structure of \( \widetilde{\mathcal{E}} \) is similar, for \( f \) corresponding to \( \mu - \omega \) and \( g \) to \( \psi_q \). 
However, in this case we need to give a proper definition of Fourier transform for the difference of probability measures as well as for the polynomially growing function $\psi_q$, 
whose distributional Fourier transform does not match the pairing occurring in the final integral.  
For that we need first to recall the definition of Fourier transform of a measure, given by
\begin{equation*}
  \widehat{\mu}(\xi) = \int_{\mathbb{R}^d} \exp \left( -i x \cdot \xi \right) \diff \mu(x),
\end{equation*}
and to use  the notion of \emph{generalized Fourier transform} from \cite[Definition 8.9]{Wend05} for the function $\psi_q$. 
The statements needed here  can be summarized in the following lemma,
which is adapted from \cite[Theorem 8.15]{Wend05}.
\begin{lemma}[Generalized Fourier transform of power functions] 
 \label{lem:1}
  Let $1<q<2$ and \( \gamma \in \mathcal{S}(\mathbb{R}^d) \) such that
  \begin{equation}
    \label{eq:4}
    \gamma(\xi) = O(\left| \xi \right|^2) \quad \text{for } \xi \to 0
  \end{equation}
  for \( m = \lceil 2 q \rceil \). 
  Then
  \begin{equation*}
    \int_{\mathbb{R}^d} \widehat{\gamma}(x)\,|x|^q \diff x = 2(2\pi)^d D_q\int_{\mathbb{R}^d} \gamma(\xi) \, |\xi|^{-(q+d)} \diff \xi
  \end{equation*}
  with the constant
  \begin{align*}
    D_q := -(2 \pi)^{-d/2} \frac{2^{q+d/2} \, \Gamma((d+q)/2)}{2 \Gamma(-q/2)} > 0.
  \end{align*}
\end{lemma}
The fact that \( \mu \) and \( \omega \) both have mass \( 1 \) corresponds to the correct decay in their Fourier transforms, satisfying condition \eqref{eq:4}. Via approximation arguments, we can derive the following:
\begin{proposition}
  \label{prp:fourier-repr}
  \cite[Corollary 3.6]{13-FornasierHuetter-VarProperties}
  Let \( \omega \in \mathcal{P}_2(\mathbb{R}^d) \) 
  and denote
  \begin{align*}
   \widehat{\mathcal{E}}[\mu] := D_q \int_{\mathbb{R}^d} \left| \widehat{\mu}(\xi) - \widehat{\omega}(\xi) \right|^2 \, \left| \xi \right|^{-d-q} \diff \xi, \quad \mu \in \mathcal{P}_2(\mathbb{R}^d).
  \end{align*}
  Then
  \begin{equation*}
    \widetilde{\mathcal{E}}[\mu] = \widehat{\mathcal{E}}[\mu], \quad \text{for all $\mu \in \mathcal{P}_2(\mathbb{R}^d)$}.
  \end{equation*}
\end{proposition}
We remark that \( \widehat{\mathcal{E}} \) coincides with the lower semi-continuous envelope of \( \widetilde{\mathcal{E}} \) on $\mathcal{P}(\mathbb{R}^d)$.
See \cite[Corollary 3.10]{13-FornasierHuetter-VarProperties} for details.

If we now combine Lemma \ref{lem:1} and the representation of Proposition \ref{prp:fourier-repr}, we get the desired moment bound for the sublevels of \( \widehat{\mathcal{E}} \).
\begin{theorem}[Moment bound]
  \label{thm:moment-bound}
  \cite[Theorem 4.1]{13-FornasierHuetter-VarProperties}
  Let \( \omega \in \mathcal{P}_2(\mathbb{R}^d) \). 
  For \( 0 < r < q/2 \), the functional \( \widehat{\mathcal{E}} \) has uniformly bounded \( r \)th moments, \ie, 
  for each given \( M \geq 0 \), there exists an \( M' \geq 0 \) such that
  \begin{equation}
    \label{eq:10}
    \int_{\mathbb{R}^d} \left| x \right|^r \diff \mu(x) \leq M', \quad \text{for all } \mu \in \mathcal P(\mathbb R^d) \text{ such that } \widehat{\mathcal{E}}[\mu] \leq M.
  \end{equation}
\end{theorem}
\begin{proof}[Sketch of the proof]
  By the assumptions that \( \widehat{\mathcal{E}}[\mu] \leq M \) and \( \omega \in \mathcal{P}_2(\mathbb{R}^d) \), 
  and from the estimate
  \begin{equation*}
    \left| \widehat{\mu}(\xi) - 1 \right|^2 \le 2 \left| \widehat{\mu}(\xi) - \widehat{\omega}(\xi) \right|^2 + 2 \left| \widehat{\omega}(\xi) - 1 \right|^2 
  \end{equation*}
  we deduce the bound
  \begin{equation*}
    \int_{\mathbb{R}^d} \left| \widehat{\mu}(\xi) - 1 \right|^2 \, \left| \xi \right|^{-d-q} \diff \xi \leq M'' := 2\left(D_q^{-1} M +  \int_{\mathbb{R}^d} \left| 1-\widehat{\omega}(\xi) \right|^2 \, \left| \xi \right|^{-d-q} \diff \xi \right).
  \end{equation*}
  We now want to use Lemma \ref{lem:1}, where we formally set \( \widehat{\gamma} =\mu - \delta_0 \) or \( \gamma =\widehat{\mu} - 1 \). Of course, in general this \( \gamma  \) will not be in \( \mathcal{S}(\mathbb{R}^d) \), 
but let us for the moment argue that by an approximation argument one could extend Lemma \ref{lem:1} also to differences of probability measures. Then we would have formally the following estimates
  \begin{align}
    \int_{\mathbb{R}^d} \left| x \right|^r \diff \mu(x) = {} & \int_{\mathbb{R}^d} \left| x \right|^{r} \widehat{\gamma}(x) \diff x \nonumber \\
    = {} & C \int_{\mathbb{R}^d} \left| \xi \right|^{-d-r} \gamma(\xi) \diff \xi \qquad \text{(Lemma \ref{lem:1})} \nonumber \\
    \leq {} & C \Bigg[ \int_{\left| \xi \right| \leq 1} \underbrace{\left| \xi \right|^{-d-r}}_{\text{{\(= \left| \xi \right|^{-\frac{d-q+2r}{2}} \left| \xi \right|^{-\frac{d + q}{2}}\)}}} \left| \gamma(\xi) \right| \diff \xi + \underbrace{\int_{\left| \xi \right| > 1} \left| \xi \right|^{-d-r} \left| \gamma(\xi) \right| \diff \xi}_{\leq C < \infty} \Bigg]\label{eq:14}\\
    \leq {} & C \Bigg[ \underbrace{\left( \int_{\left| \xi \right| \leq 1} \left| \xi \right|^{-d+(q - 2r)} \diff \xi \right)^{1/2}}_{\smash{< \infty}} \left( \int_{\mathbb{R}^d} \left| \xi \right|^{-d-q} \left| \gamma(\xi) \right|^2 \diff \xi \right)^{1/2} + 1 \Bigg] \nonumber \\
    \leq {} & C \left[ \left( \int_{\mathbb{R}^d} \left| \xi \right|^{-d-q} \left| \gamma(\xi) \right|^2 \diff \xi \right)^{1/2} + 1 \right] \leq C\big((M'')^{1/2}+1\big), \nonumber
  \end{align}
  yielding the desired bound, where in \eqref{eq:14} we used H\"older’s inequality. This computation can  be made indeed rigorous by appropriate approximation arguments, for which we refer  to the proof of \cite[Theorem 4.1]{13-FornasierHuetter-VarProperties}.
\end{proof}

\section*{Acknowledgements}
MDF acknowledges the hospitality at the Zentrum Mathematik of the Technische Universität München during his visits. 
MDF is supported by the FP7-People Marie Curie CIG (Career Integration Grant) 
Diffusive Partial Differential Equations with Nonlocal Interaction in Biology and Social Sciences (DifNonLoc),
by the ``Ramon y Cajal'' sub-programme (MICINN-RYC) of the Spanish Ministry of Science and Innovation, Ref. RYC-2010-06412, 
and by the Ministerio de Ciencia e Innovaci\'{o}n, grant MTM2011-27739-C04-02.

MF is supported by the ERC-Starting Grant, project 306274-HDSPCONTR  ``High-Dimensional Sparse Optimal Control''.

MF's and DM's research was supported by the DFG Collaborative Research Center TRR 109, ``Discretization in Geometry and Dynamics''.

\bibliographystyle{abbrv}
\bibliography{BibKineticDitheringGF}

\begin{thebibliography}{10}

\bibitem{AGS08}
L.~Ambrosio, N.~Gigli, and G.~Savar{\'e}.
\newblock {\em Gradient Flows in Metric Spaces and in the Space of Probability
  Measures}.
\newblock Lectures in Mathematics ETH Z\"urich. Birkh\"auser Verlag, Basel,
  second edition, 2008.

\bibitem{BCLR}
D.~Balague, J.~A. Carrillo, T.~Laurent, and G.~Raoul.
\newblock Nonlocal interactions by repulsive-attractive potentials: radial
  ins/stability.
\newblock {\em arxiv.org/abs/1109.5258v1}, 2012.

\bibitem{belaro11}
A.~L. Bertozzi, T.~Laurent, and J.~Rosado.
\newblock {$L^p$} theory for the multidimensional aggregation equation.
\newblock {\em Comm. Pure Appl. Math.}, 64(1):45--83, 2011.

\bibitem{68-Billingsley-Conv-Proba}
P.~Billingsley.
\newblock {\em Convergence of Probability Measures}.
\newblock John Wiley \& Sons Inc., New York, 1968.

\bibitem{95-Billingsley-Proba_and_Measuer}
P.~Billingsley.
\newblock {\em Probability and Measure}.
\newblock Wiley Series in Probability and Mathematical Statistics. John Wiley
  \& Sons Inc., New York, third edition, 1995.
\newblock A Wiley-Interscience Publication.

\bibitem{BS}
G.~A. Bonaschi.
\newblock {Gradient Flows Driven by a Non-Smooth Repulsive Interaction
  Potential}.
\newblock Master's thesis, Universit\`{a} di Pavia, 2011.

\bibitem{BCDP}
G.~A. Bonaschi, J.~A. Carrillo, M.~Di~Francesco, and M.~A. Peletier.
\newblock Equivalence of gradient flows and entropy solutions for singular
  nonlocal interaction equations in 1d.
\newblock {\em arXiv:1310.4110}, 2013.

\bibitem{BDi08}
M.~Burger and M.~Di~Francesco.
\newblock Large time behavior of nonlocal aggregation models with nonlinear
  diffusion.
\newblock {\em Netw. Heterog. Media}, 3(4):749--785, 2008.

\bibitem{13-Carrillo-Choi-Hauray-MFL}
J.~A. Carrillo, Y.-P. Choi, and M.~Hauray.
\newblock The derivation of swarming models: mean-field limit and {W}asserstein
  distances.
\newblock {\em arXiv:1304.5776}, 2013.

\bibitem{CDFLS}
J.~A. Carrillo, M.~DiFrancesco, A.~Figalli, T.~Laurent, and D.~Slep{\v{c}}ev.
\newblock Global-in-time weak measure solutions and finite-time aggregation for
  nonlocal interaction equations.
\newblock {\em Duke Math. J.}, 156(2):229--271, 2011.

\bibitem{CLM}
J.~A. Carrillo, S.~Lisini, and E.~Mainini.
\newblock Gradient flows for non-smooth interaction potentials.
\newblock {\em arXiv:1206.4453}, 2012.

\bibitem{07_Carillo_Toscani_prob-metrics}
J.~A. Carrillo and G.~Toscani.
\newblock Contractive probability metrics and asymptotic behavior of
  dissipative kinetic equations.
\newblock {\em Riv. Mat. Univ. Parma (7)}, 6:75--198, 2007.

\bibitem{FR2}
K.~Fellner and G.~Raoul.
\newblock Stability of stationary states of non-local equations with singular
  interaction potentials.
\newblock {\em Mathematical and Computer Modelling}, 53(7-8), 2011.

\bibitem{FR1}
K.~Fellner and G.~Raoul.
\newblock Stable stationary states of non-local interaction equations.
\newblock {\em Mathematical Models and Methods in Applied Sciences}, 20(12),
  2011.

\bibitem{FHS12}
M.~Fornasier, J.~Ha\v{s}kovec, and G.~Steidl.
\newblock {Consistency of variational continuous-domain quantization via
  kinetic theory}.
\newblock {\em Appl. Anal.}, 92(6):1283--1298, 2013.

\bibitem{13-FornasierHuetter-VarProperties}
M.~Fornasier and J.-C. H{\"u}tter.
\newblock {Consistency of probability measure quantization by means of power
  repulsion-attraction potentials}.
\newblock {\em arXiv/1310.1120}, 2013.

\bibitem{grpost11}
M.~Gr\"af, D.~Potts, and G.~Steidl.
\newblock {Quadrature errors, discrepancies, and their relations to halftoning
  on the torus and the sphere.}
\newblock {\em SIAM J. Sci. Comput.}, 34(5):A2760--A2791, 2012.

\bibitem{piro13}
B.~Piccoli and F.~Rossi.
\newblock Generalized {W}asserstein distance and its application to transport
  equations with source.
\newblock {\em Arch. Ration. Mech. Anal.}, (DOI) 10.1007/s00205-013-0669-x.

\bibitem{ra10}
G.~Raoul.
\newblock {Nonlocal interaction equations: Stationary states and stability
  analysis.}
\newblock {\em Differ. Integral Equ.}, 25(5-6):417--440, 2012.

\bibitem{scgwbrwe11}
C.~Schmaltz, G., P.~Gwosdek, A.~Bruhn, and J.~Weickert.
\newblock {Electrostatic halftoning.}
\newblock {\em Computer Graphics}, 29:2313--2327, 2010.

\bibitem{testgwscwe11}
T.~Teuber, G.~Steidl, P.~Gwosdek, C.~Schmaltz, and J.~Weickert.
\newblock {Dithering by differences of convex functions}.
\newblock {\em SIAM J. Imaging Sci.}, 4(1):79--108, 2011.

\bibitem{Wend05}
H.~Wendland.
\newblock {\em Scattered Data Approximation}, volume~17 of {\em Cambridge
  Monographs on Applied and Computational Mathematics}.
\newblock Cambridge University Press, Cambridge, 2005.

\end{thebibliography}

\end{document}